\documentclass[12pt]{article}
\usepackage{layout}
\usepackage{url}
\usepackage{proof}
\usepackage{mathrsfs}
\usepackage{amsmath,amsthm,amssymb}
\usepackage{cleveref}
\usepackage[dvipdfmx]{graphicx}
\usepackage[all]{xy}
\usepackage{color}
\usepackage{indentfirst}

\numberwithin{equation}{section}

\newtheorem{thm}{Theorem}[section]
\newtheorem{prop}[thm]{Proposition}
\newtheorem{lemm}[thm]{Lemma}
\newtheorem{cor}[thm]{Corollary}
\newtheorem{defi}[thm]{Definition}

\newtheorem{rmk}[thm]{Remark}

\def\fk#1{\mathfrak{#1}}
\def\cal#1{\mathcal{#1}}

\def\rm#1{\mathrm{#1}}
\def\bb#1{\mathbb{#1}}

\def\wh#1{\widehat{#1}}
\def\wt#1{\widetilde{#1}}
\def\ol#1{\overline{#1}}

\newcommand{\QQ}{\mathbb{Q}}
\newcommand{\Qlb}{\overline{\mathbb{Q}}_\ell}

\newcommand{\GL}{\mathrm{GL}}
\newcommand{\Hom}{\mathrm{Hom}}

\newcommand{\Perv}{\mathrm{Perv}}

\newcommand{\Sat}{\mathrm{Sat}}

\newcommand{\Hck}{\mathcal{H}\mathrm{ck}}
\newcommand{\AHck}{\mathcal{AH}\mathrm{ck}}
\newcommand{\HHom}{\mathscr{H}\mathrm{om}}
\newcommand{\Div}{\mathrm{Div}}
\newcommand{\Divtil}{\widetilde{\mathrm{Div}}_0^{\perp}}
\newcommand{\ULA}{\mathrm{ULA}}
\newcommand{\perf}{\mathrm{perf}}
\newcommand{\Perf}{\mathrm{Perf}}

\newcommand{\et}{\mathrm{\acute{e}t}}

\newcommand{\bd}{\mathrm{bd}}

\newcommand{\dia}{\diamondsuit}
\newcommand{\Sym}{\mathrm{Sym}}
\newcommand{\Witt}{\mathrm{Witt}}
\newcommand{\eq}{\mathrm{eq}}

\newcommand{\Gr}{\mathrm{Gr}}
\newcommand{\Fl}{\mathcal{F}\ell}

\newcommand{\relmiddle}[1]{\mathrel{}\middle#1\mathrel{}}
\DeclareMathOperator{\Spec}{Spec}
\DeclareMathOperator{\Spd}{Spd}
\DeclareMathOperator{\Spa}{Spa}

\DeclareMathOperator{\supp}{supp}

\title{\Large{Derived Satake category and Affine Hecke category in mixed characteristics}}
\author{Katsuyuki Bando\footnote{Katsuyuki Bando, the University of Tokyo, Graduate School of Mathematical Sciences, Japan, \texttt{kbando@ms.u-tokyo.ac.jp}}}

\begin{document}
\maketitle
\abstract{
We construct a new affine Grassmannian which connects an equal characteristic affine Grassmannian and Zhu's Witt vector affine Grassmannian.
As a result, we deduce the mixed characteristic version of the Bezrukavnikov--Finkelberg's derived Satake equivalence.
By the same argument, we also obtain the mixed characteristic version of Bezrukavnikov's equivalence between two categorifications of an affine Hecke algebra.
}
\section{Introduction}
A lot of analogies between equal characteristic local fields and mixed characteristic local fields have been studied for a long time.
Therefore, it is natural to expect analogies between an equal characteristic affine Grassmannian and a mixed characteristic one.
However, we are often faced with technical difficulties in mixed characteristic.

In this paper, we construct a version of an affine Grassmannian which contains both an equal characteristic affine Grassmannian and a mixed characteristic one.
Through this construction, we solve several mixed characteristic problems.

Put $k=\overline{\bb{F}}_p$, and let $G$ be a reductive group over $k[[t]]$.
Write $\wh{G}$ for its dual group over $\Qlb$, and $\wh{\fk{g}}$ its Lie algebra.

According to \cite{BF}, there is a canonical equivalence of monoidal triangulated categories, which is called derived geometric Satake correspondence:
\[
D^b_{L^{+,\eq}G}(\Gr_{G,\Spec k}^{\eq},\Qlb)\overset{\sim}{\longrightarrow} D^{\wh{G}}_{perf}(\Sym^{[]}(\wh{\fk{g}})),
\]
where the left-hand side is the bounded constructible derived category of loop group-equivariant complexes on the affine Grassmannian in equal characteristics, and the right-hand side is the category of $\wh{G}$-equivariant perfect complexes on the differential graded version of $\Sym(\wh{\fk{g}})$, (for details, see \cite{BF}).
We also write the left-hand side as
\[
D^b(\Hck_{G,\Spec k}^{\eq},\Qlb), 
\]
where $\Hck_{G,\Spec k}^{\eq}:=[L^+G^{\eq}\backslash \Gr_{G,\Spec k}^{\eq}]$ is called the Hecke stack in equal characteristics.

We want a mixed characteristic version of this equivalence, using a Witt-vector affine Grassmaniann defined in \cite{Zhumixed}.
However, the proof in \cite{BF} does not directly work in mixed characteristics.
The problem is, for example, that there does not exist a loop rotation action of $\bb{G}_m$ on a loop group $L^{+}G$.

The main result of this paper is the following equivalence connecting mixed characteristics with equal characteristics:
\begin{thm}(Theorem \ref{thm:maineqmixedtorsion})\label{thm:DerivedHckeqmixed}
Let $\Lambda$ be either
\begin{enumerate}
\item a ring killed by some power of $\ell$, 
\item an algebraic extension $L$ of $\bb{Q}_{\ell}$, or
\item its ring of integer $\cal{O}_L$.
\end{enumerate}
Let $E$ be a finite extension of $\QQ_p$, and $G$ a reductive group over $\cal{O}_E[[t]]$.
There is a canonical equivalence of monoidal triangulated categories
\[
D^{\ULA}(\Hck_{G,\Spec k}^{\eq},\Lambda)\simeq D^{\ULA}(\Hck_{G,\Spec k}^{\Witt},\Lambda).
\]
In particular, when $\Lambda=\Qlb$, we have
\[
D^b(\Hck_{G,\Spec k}^{\eq},\Qlb)\simeq D^b(\Hck_{G,\Spec k}^{\Witt},\Qlb).
\]
\end{thm}
Here $\Hck_{G,\Spec k}^{\Witt}$ is the Witt-vector Hecke stack, which is a quotient stack of the Witt-vector affine Grassmaniann, defined in \cite{Zhumixed} by the Witt-vector loop group.
As a result, we get the derived geometric Satake correspondence in mixed characteristics:
\begin{thm}[Derived geometric Satake correspondence in mixed characteristics](Corollary \ref{cor:DerivedSatbody})
Let $E$ be a finite extension of $\QQ_p$, and $G$ a reductive group over $\cal{O}_E$.
There is a canonical equivalence of monoidal triangulated categories
\[
D^b(\Hck_{G,\Spec k}^{\Witt},\Qlb)\overset{\sim}{\longrightarrow} D^{\wh{G}}_{perf}(\Sym^{[]}(\wh{\fk{g}})).
\]
\end{thm}
To show Theorem \ref{thm:DerivedHckeqmixed}, we connect an equal characteristic problem with a mixed characteristic one.
Our method is different from tilting:
\[
\xymatrix{
\Div^{1,\eq}&&\Div^{1,\mathrm{mixed}}\\
\Spd k\ar@{->}[r]_-{(\pi,t)}\ar[u]&(\Divtil)^{\dia}&\Spd k.\ar@{->}[l]^-{(t,\pi)}\ar[u]
}
\]
The moduli space $\Div^{1,\mathrm{mixed}}$ of untilts and its equal characteristic correspondent $\Div^{1,\eq}$ are in a different direction from the space $\Divtil$ we are going to use.

Instead, we construct a version of an affine Grassmannian $\Gr_{G,\Divtil}$ over a certain scheme $\Divtil$ connecting equal characteristics and mixed characteristics.
More precisely, the fiber of $\Gr_{G,\Divtil}$ over a certain point of $\Divtil$ is an equal characteristic affine Grassmannian, and the fiber over another point is a Witt vector affine Grassmannian.

Let $E$ be as above and $k_E$ a residue field of $\cal{O}_E$.
The scheme $\Divtil$ is defined as a presheaf on the category of perfect $k_E$-schemes
\[
R\mapsto \left\{(\xi_1,\xi_2)\in W_{\cal{O}_E}(R)[[t]]^2\relmiddle| \begin{array}{l}\xi_1=[a]p+[b]t\\
\xi_2=[c]p+[d]t\end{array},\begin{pmatrix}a&b\\c&d\end{pmatrix}\in \GL_2(R)\right\},
\]
where $[-]$ denotes the Teichm\"{u}ller lift.
This scheme $\Divtil$ is related to a moduli space of divisors on ``Fargues--Fontaine surfaces'', which connects an equal characteristic Fargues--Fontaine curve and a mixed characteristic one.
That is, for a perfectoid space $S=\Spa(R,R^+)$ of characteristic $p$, put
\[
\underline{\cal{Y}}_S:=(\Spa W_{\cal{O}_E}(R^+)[[t]])\setminus V([\varpi]), 
\]
where $\varpi$ is a pseudo uniformizer of $R$.
Consider a moduli space $\Div^{1,1}$ (which is a v-sheaf on the category of perfectoid spaces over $k_E$) defined by 
\begin{align*}
&(R,R^+)\\
&\mapsto \left\{(D_1,D_2)\relmiddle|\begin{array}{l}
\text{$D_1$ is a closed Cartier divisor of $\underline{\cal{Y}}_S$.}\\
\text{$D_2$ is a closed Cartier divisor of (a space defined by) $D_1$.}\\
\text{$D_2$ is isomorphic to some $\bb{Z}_p[[t]]$-untilt $S^{\sharp}$ of $S$,}\\
\text{ and the induced map $S \cong D_2^{\dia} \hookrightarrow \underline{\cal{Y}}_S^{\dia}$ is a map over $S\times \Spd \bb{Z}_p[[r]]$.}
\end{array} \right\}
\end{align*}
Then a certain quotient of $\Divtil$ can be seen as a subspace of $\Div^{1,1}$.

Using this moduli space reduces the mixed characteristic statement to the equal characteristic one.
Here we use certain equivalence of the category of ULA sheaves, similarly to \cite[Corollary VI.6.7]{FS}.

Furthermore, according to \cite{Bez}, there is a canonical equivalence of monoidal categories:
\begin{align}\label{eqn:BezQlb}
D^b_{L^{+,\eq}\cal{I}}(\Fl_{G,\Spec k}^{\eq},\Qlb)\overset{\sim}{\longrightarrow} \mathrm{DGCoh}^{\wh{G}}(\wt{\cal{N}}\times_{\wh{\fk{g}}}\wt{\cal{N}}),
\end{align}
where the left-hand side is the bounded constructible derived category of $L^+\cal{I}^{\eq}$-equivariant complexed on the affine flag variety in equal characteristics, and the right-hand side is the category of $\wh{G}$-equivariant differential graded coherent complex on the differential graded scheme $\wt{N}\times_{\wh{\fk{g}}}\wt{N}$ over $\Qlb$ (for details, see \cite{Bez}).
We also write the left-hand side as
\[
D^b(\AHck_{G,\Spec k}^{\eq},\Qlb), 
\]
where $\AHck_{G,\Spec k}^{\eq}:=[L^+\cal{I}^{\eq}\backslash \Fl_{G,\Spec k}^{\eq}]$ is here called the affine Hecke stack in equal characteristics.

The proof of (\ref{eqn:BezQlb}) in \cite{Bez} also does not work in mixed characteristics.
The problem is, for example, that we do not know Gaitsgory's central functor \cite{Gai} can be defined between the category of perverse sheaves, even if we use the mixed characteristic Beilinson--Drinfeld affine Grassmannian in \cite{FS}.
We do not know the pushforward functor by an open immersion of a generic point of a spectrum of a valuation ring preserves the perversity in the context of diamonds.

However, we also show an analogous result to Theorem \ref{thm:DerivedHckeqmixed}:
\begin{thm}(Theorem \ref{thm:maineqmixedtorsion})\label{thm:DerivedAHckeqmixed}
Let $\Lambda, E, G$ be as in Theorem \ref{thm:DerivedHckeqmixed}.
There is a canonical equivalence of monoidal triangulated categories
\[
D^{\ULA}(\AHck_{G,\Spec k}^{\eq},\Lambda)\simeq D^{\ULA}(\AHck_{G,\Spec k}^{\Witt},\Lambda).
\]
In particular, when $\Lambda=\Qlb$, we have
\[
D^b(\AHck_{G,\Spec k}^{\eq},\Qlb)\simeq D^b(\AHck_{G,\Spec k}^{\Witt},\Qlb).
\]
\end{thm}
Here $\AHck_{G,\Spec k}^{\Witt}$ is the Witt-vector affine Hecke stack, which is a quotient stack of the Witt-vector affine flag variety, defined in \cite[\S 1.4]{Zhumixed}.
As a result, we obtain a mixed characteristic version of the equivalence of \cite{Bez}.
\begin{thm}(Corollary \ref{cor:AffHckbody})
Let $E$ be a finite extension of $\QQ_p$, and $G$ a reductive group over $\cal{O}_E$.
There is a canonical equivalence of monoidal triangulated categories
\[
D^b(\AHck_{G,\Spec k}^{\Witt},\Qlb)\overset{\sim}{\longrightarrow} \mathrm{DGCoh}^{\wh{G}}(\wt{\cal{N}}\times_{\wh{\fk{g}}}\wt{\cal{N}}).
\]
\end{thm}

After we uploaded the first version of this article to arXiv, a related article \cite{ALWY} appeared.
They prove the equivalence for $G$ of type $A$ by showing some results on Gaitsgory's central functor (\cite[Theorem 1.5]{ALWY}).
Our proof is different in that we use the moduli space of ``Fargues--Fontaine surface'', connecting mixed characteristics with equal characteristics, which can be applied for general $G$.

Furthermore, by proving the exactness of the equivalence of Theorem \ref{thm:DerivedHckeqmixed} and Theorem \ref{thm:DerivedAHckeqmixed}, we show the following theorem:
\begin{thm}(Theorem \ref{thm:equivPerv})\label{thm:SatHckeqmixed}
Let $\Lambda, E, G$ be as in Theorem \ref{thm:DerivedHckeqmixed}.
There is a canonical equivalence of monoidal categories
\begin{align}\label{eqn:SatHckeqmixed}
\Sat(\Hck_{G,\Spec k}^{\eq},\Lambda)\simeq \Sat(\Hck_{G,\Spec k}^{\Witt},\Lambda).
\end{align}
In particular, when $\Lambda=\Qlb$, we have
\[
\Perv(\Hck_{G,\Spec k}^{\eq},\Qlb)\simeq \Perv(\Hck_{G,\Spec k}^{\Witt},\Qlb).
\]
\end{thm}
Here $\Sat$ means the full subcategory of the derived category consisting of flat perverse ULA sheaves.
We also prove the following:
\begin{thm}(Theorem \ref{thm:symmonoidal})
The equivalence (\ref{eqn:SatHckeqmixed}) is symmetric monoidal.
\end{thm}
To prove this, we need the construction of a symmetric monoidal structure via fusion product, thus we use some diamonds and v-stacks, including the moduli space $\Div^{1,1}$ of divisors on ``Fargues--Fontaine surface''.
\subsection*{Acknowledgement}
The author is really grateful to his advisor Naoki Imai for useful discussions and comments.
He also would like to thank Peter Scholze and Jo\~{a}o Louren\c{c}o for their helpful comments.
\subsection*{Notations}
Let $p$ be a prime number, and $E$ a finite extension of $\QQ_p$.
We write $\cal{O}_E$ for its ring of integers, and $k_E$ for its residue field.
Also, let $\varpi_E$ be a uniformizer of $E$.
Let $\ell$ be a prime number not equal to $p$.
Let $\Lambda$ be either
\begin{enumerate}
\item a ring killed by some power of $\ell$, 
\item an algebraic extension $L$ of $\bb{Q}_{\ell}$, or
\item its ring of integer $\cal{O}_L$.
\end{enumerate}

Let $G$ be a connected reductive group over $\cal{O}_E[[t]]$.
Write $\bar{G}$ for $G\times_{\Spec \cal{O}_E[[t]]} \Spec k_E$
When $G$ is split, $T$ denotes a maximal split torus in $G$, and $B$ denotes a Borel subgroup containing $T$.
Write $\bb{X}^*(T)=\Hom(T,\bb{G}_m)$ for the character lattice, $\bb{X}_*(T)=\Hom(\bb{G}_m,T)$ the cocharacter lattice.
Let $\bb{X}_*^+(T)$ denote the subset of dominant cocharacters with respect to $B$.

For a field $k$, let $\rm{Alg}_{k}$ be the category of $k$-algebras, and when $k$ is a perfect field of characteristic $p$, let $\rm{PerfAlg}_{k}$ be the category of perfect $k$-algebras.

For a ring $R$ over $\bb{F}_p$, its perfection is denoted by $R^{\perf}$, that is, 
\[
R^{\perf}=\lim_{\substack{\longrightarrow \\ \sigma}} R
\]
where $\sigma\colon R\to R$ is a homomorphism $a\mapsto a^p$.
Similarly, for a scheme $X$, its perfection is denoted by $X^{\perf}$.
If $k$ is a perfect field of characteristic $p$ and $X$ has a $k$-structure, $k$-structure on $X^{\perf}$ is defined by the composition
\[
X^{\perf}\to X\to \Spec k.
\]
In other words, $k$-scheme $X^{\perf}$ can be written as an inverse limit
\[
X=\lim_{\longleftarrow}(X\overset{\sigma}{\leftarrow} X^{(1)}\overset{\sigma}{\leftarrow}X^{(2)}\overset{\sigma}{\leftarrow}\cdots)
\]
where $X^{(m)}=X$ with a $k$-structure given by 
\[
X\to \Spec k \overset{\sigma^m}{\to} \Spec k.
\]

\section{Affine Grassmannians and Affine flag varieties}
\subsection{Review on equal characteristic and Witt vector affine Grassmannians}\label{ssc:reviewGrass}
Put 
\begin{align*}
L^{+,\eq}G&\colon \rm{Alg}_{k_E}\to \rm{Sets},\ R\mapsto G(R[[t]]),\\
L^{\eq}G&\colon \rm{Alg}_{k_E}\to \rm{Sets},\ R\mapsto G(R((t))).
\end{align*}
The equal characteristic affine Grassmannian is a quotient stack in \'{e}tale topology
\[
\Gr_{G,\Spec k_E}^{\eq}:=[L^{\eq}G/L^{+,\eq}G].
\]
It is known that $\Gr_{G,\Spec k_E}^{\eq}$ is representable by an inductive limit of projective $k$-schemes along closed immersions, see \cite{MV}.
Moreover, put
\begin{align*}
L^{+,\Witt}G&\colon \rm{PerfAlg}_{k_E}\to \rm{Sets},\ R\mapsto G(W_{\cal{O}_E}(R)),\\
L^{\Witt}G&\colon \rm{PerfAlg}_{k_E}\to \rm{Sets},\ R\mapsto G(W_{\cal{O}_E}(R)[\varpi_E^{-1}]),
\end{align*}
where 
\[
W_{\cal{O}_E}(R)=W(R)\otimes_{W(k_E)}\cal{O}_E
\]
is a ring of Witt vector tensored with $\cal{O}_E$.
The Witt vector affine Grassmannian is a quotient stack in \'{e}tale topology
\[
\Gr_{G,\Spec k_E}^{\Witt}:=[L^{\Witt}G/L^{+,\Witt}G].
\]
$\Gr_{G,\Spec k}$ is representable by an inductive limit of the perfections of projective $k$-schemes along closed immersions, see \cite{BS}.
\subsection{Affine Grassmannian over $\Divtil$}
\subsubsection{Definiton of affine Grassmannian}
In this subsection, we define a version of the affine Grassmannian connecting the equal characteristic affine Grassmannian and the Witt vector affine Grassmannian.
Fix a uniformizer $\pi$ in $E$ and define the presheaf
\[
\Divtil\colon \rm{PerfAlg}_{k_E}\to \rm{Sets}
\]
by 
\[
\Divtil(R)=\left\{(\xi_1,\xi_2)\in W_{\cal{O}_E}(R)[[t]]^2\relmiddle| \begin{array}{l}\xi_1=[a]\pi+[b]t\\
\xi_2=[c]\pi+[d]t\end{array},\begin{pmatrix}a&b\\c&d\end{pmatrix}\in \GL_2(R)\right\},
\]
where $[-]$ denotes the Teichm\"{u}ller lift.
Obviously, there is a canonical isomorphism
\[
\GL_{2,k_E}^{\perf}\overset{\sim}{\longrightarrow}\Divtil, \begin{pmatrix}a&b\\c&d\end{pmatrix}\mapsto ([a]\pi+[b]t,[c]\pi+[d]t).
\]

For a ring $R\in \rm{PerfAlg}_{k_E}$ and $(\xi_1,\xi_2)\in \Divtil(R)$, put
\begin{align*}
B^+_{\xi_1,\xi_2}(R)&:=(W_{\cal{O}_E}(R)[[t]]/(\xi_1))^{\wedge \xi_2}\\
B_{\xi_1,\xi_2}(R)&:=(W_{\cal{O}_E}(R)[[t]]/(\xi_1))^{\wedge \xi_2}[\xi_2^{-1}], 
\end{align*}
where $(-)^{\wedge \xi_2}$ is the $\xi_2$-adic completion.
We still write $\xi_2$ for the image of $\xi_2$ in $W_{\cal{O}_E}(R)[[t]]/(\xi_1)$ or $B^+_{\xi_1,\xi_2}(R)$.
We define
\begin{align*}
L^{+}_{\Divtil}G&\colon \rm{PerfAffSch}^{\rm{op}}_{/\Divtil}\to \rm{Sets},\ \Spec R\mapsto G(B^+_{\xi_1,\xi_2}(R)),\\
L_{\Divtil}G&\colon \rm{PerfAffSch}^{\rm{op}}_{/\Divtil}\to \rm{Sets},\ \Spec R\mapsto G(B_{\xi_1,\xi_2}(R)).
\end{align*}
We sometimes write simply $L^+G$ and $LG$ for $L^{+}_{\Divtil}G$ and $L_{\Divtil}G$, respectively.
\begin{defi}
{The affine Grassmannian over $\Divtil$} is defined as the following quotient stack in \'{e}tale topology:
\[
\Gr_{G,\Divtil}:=[L_{\Divtil}G/L^+_{\Divtil}G].
\]
\end{defi}
For $\Gr_{G,\Divtil}$, we sometimes write $\Gr_G$ or $\Gr$ if $G$ is clear from the context.
\begin{rmk}
The space $\Gr_{G,\Divtil}$ depends on the choice of $\pi$.
While we use this space to show the main theorem \ref{thm:DerivedHckeqmixed}, we can show that the equivalence of Theorem \ref{thm:DerivedHckeqmixed} does not depend on the choice of $\pi$.
In fact, $\Gr_{G,\Divtil}$ can be seen as the pullback of a closed subscheme of a certain affine Grassmanninan which does not depend on the choice of $\pi$ (and also in which Teichm\"{u}ller lifts are not involved), see \S \ref{scn:symm}.
\end{rmk}
\begin{rmk}
For $(\xi_1,\xi_2)\in \Divtil(R)$, the ring $W_{\cal{O}_E}(R)[[t]]/(\xi_1)$ is already $\xi_2$-complete, thus $L^{+}_{\Divtil}G$, $L_{\Divtil}G$ and $\Gr_{G,\Divtil}$ can be already defined over $(\bb{A}^{2}\setminus \{0\})^{\perf}$ via 
\[
\Divtil\to (\bb{A}^{2}\setminus \{0\})^{\perf},\ (\xi_1,\xi_2)\mapsto \xi_2.
\]
(Furthermore, they are defined over $\bb{P}^{1,\perf}$.)
However, we define as above to clarify the relation to the moduli space of divisors on Fargues--Fontaine surface, see \S \ref{scn:symm}.
\end{rmk}

In terms of $G$-torsors, 
\[
\Gr_{G,\Divtil}(R)=\left\{(\xi_1,\xi_2,\cal{E},\beta)\relmiddle|\begin{array}{l}
(\xi_1,\xi_2)\in \Divtil(R)\\
\text{$\cal{E}$ is a $G$-torsor on $\Spec B^+_{\xi_1,\xi_2}(R)$}\\
\text{$\beta$ is a trivialization of $\cal{E}|_{\Spec B_{\xi_1,\xi_2}(R)}$}
\end{array}\right\}.
\]
Let us state some properties of $B^+_{\xi_1,\xi_2}(R)$. 
\begin{lemm}\label{lemm:BquotR}
Let $R$ be a perfect ring over $k_E$ and fix $(\xi_1,\xi_2)\in \Divtil(R)$.
Then
\[
B^+_{\xi_1,\xi_2}(R)/\xi_2\cong R.
\]
\end{lemm}
\begin{proof}
Since $\begin{pmatrix}a&b\\c&d\end{pmatrix}\in \GL_2(R)$, we also have $\begin{pmatrix}[a]&[b]\\\text{$[c]$}&[d]\end{pmatrix}\in \GL_2(W_{\cal{O}_E}(R)[[t]])$.
Thus
\begin{align*}
B^+_{\xi_1,\xi_2}(R)/\xi_2&\cong W_{\cal{O}_E}(R)[[t]]/(\xi_1,\xi_2)\\
&= W_{\cal{O}_E}(R)[[t]]/([a]\pi+[b]t,[c]\pi+[d]t)\\
&=W_{\cal{O}_E}(R)[[t]]/(\pi,t)\\
&\cong R.
\end{align*}
\end{proof}
\begin{lemm}\label{lemm:xiadicexp}
$\xi_2\in W_{\cal{O}_E}(R)[[t]]/(\xi_1)$ is non zero divisor.

As result, together with Lemma \ref{lemm:BquotR}, any element $a\in B^+_{\xi_1,\xi_2}(R)$ is uniquely expressed as
\[
a=\sum_{i=1}^{\infty} \ol{[a_i]}\xi_2^i
\]
with $a_i\in R$.
Here $\ol{[a_i]}$ is the image in $B^+_{\xi_1,\xi_2}(R)$ of the Teichm\"{u}ler lift $[a_i]\in W_{\cal{O}_E}(R)$.
\end{lemm}
\begin{proof}
For the first part, we need to show that if 
\[
x\xi_1+y\xi_2=0
\]
for some $x,y\in W_{\cal{O}_E}(R)[[t]]$, then $x=z\xi_2, y=-z\xi_1$ for some $z\in W_{\cal{O}_E}(R)[[t]]$.
This easily follows from the analogous result for $(\pi,t)$ instead of $(\xi_1,\xi_2)$, since $(\pi,t)=(\xi_1,\xi_2)A$ for some $A\in \GL_2(W_{\cal{O}_E}(R)[[t]])$.

The second part follows from the standard argument since
\[
R\underset{\sim}{\overset{\ol{[-]}}{\longrightarrow}}B^+_{\xi_1,\xi_2}(R)/\xi_2 \underset{\sim}{\overset{\xi_2^i\times }{\longrightarrow}}  \xi_2^iB^+_{\xi_1,\xi_2}(R)/\xi_2^{i+1}B^+_{\xi_1,\xi_2}(R)
\]
is an isomorphism for every $i$.
Here the first map is bijective by Lemma \ref{lemm:BquotR} and the second is bijective by the first part.
\end{proof}
\begin{lemm}\label{lemm:Bdvr}
If $k$ is a perfect field containing $k_E$ and $(\xi_1,\xi_2)\in \Divtil(k)$, then $B^+_{\xi_1,\xi_2}(k)$ is a $\xi_2$-complete discrete valuation field and $B_{\xi_1,\xi_2}(k)$ is its fractional field.
\end{lemm}
\begin{proof}
This follows from Lemma \ref{lemm:xiadicexp}.
\end{proof}
We also need the following lemma:
\begin{lemm}
Let $(L^+_{\Divtil} G)^{\geq m}\subset L^+_{\Divtil} G$ be a closed subgroup define by the kernel of 
\[
G(B^+_{\xi_1,\xi_2}(R))\to G(B^+_{\xi_1,\xi_2}(R)/\xi_2^m).
\]
Then there are natural isomorphisms
\[
L^+_{\Divtil} G/(L^+_{\Divtil} G)^{\geq 1}\cong \bar{G}\times_{\Spec k_E} \Divtil
\]
and
\[
(L^+_{\Divtil} G)^{\geq m}/(L^+_{\Divtil} G)^{\geq m+1}\cong (\rm{Lie}(\bar{G})\times_{\Spec k_E} \Divtil)\{m\}, 
\]
where $\{m\}$ means a twist defined by $-\otimes_{R} \xi_2^mB^+_{\xi_1,\xi_2}(R)/\xi_2^{m+1} B^+_{\xi_1,\xi_2}(R)$ on $\Spec R\to \Divtil$.
\end{lemm}
\begin{proof}
The first isomorphism follows from Lemma \ref{lemm:BquotR}.
The second is also easy using Lemma \ref{lemm:xiadicexp}.
\end{proof}
\subsubsection{Schubert varieties}
Assume $G$ is split and $T\subset G$ is a maximal split torus.
For any algebraically closed field $C'$ containing $k_E$, by Lemma \ref{lemm:Bdvr} and the Cartan decomposition, we have
\[
G(B_{\xi_1,\xi_2}(C'))=\coprod_{\mu\in \bb{X}_*^+(T)} G(B^+_{\xi_1,\xi_2}(C'))\mu(\xi_2)G(B^+_{\xi_1,\xi_2}(C')),
\]
therefore
\[
\Gr_{G,\Divtil}(C')=\coprod_{\mu\in \bb{X}_*^+(T)} G(B^+_{\xi_1,\xi_2}(C'))\mu(\xi_2)G(B^+_{\xi_1,\xi_2}(C'))/G(B^+_{\xi_1,\xi_2}(C')).
\]
So we can define the Schubert varieties and Schubert cells:
\begin{defi}
\begin{enumerate}
\item For $\mu\in \bb{X}_*^+(T)$, let
\[
\Gr_{G,\Divtil,\leq \mu}\subset \Gr_{G,\Divtil}
\]
be the subfunctor of maps $\Spec R\to \Gr_{G,\Divtil,\leq \mu}$ such that for all the geometric points $\Spec C'\to \Spec R$, the corresponding $C'$-valued point is in 
\[
G(B^+_{\xi_1,\xi_2}(C'))\mu'(\xi_2)G(B^+_{\xi_1,\xi_2}(C'))/G(B^+_{\xi_1,\xi_2}(C'))
\]
 for some $\mu'$ with $\mu'\leq \mu$.
\item For $\mu\in \bb{X}_*^+(T)$, let
\[
\Gr_{G,\Divtil, \mu}\subset \Gr_{G,\Divtil}
\]
be the subfunctor of maps $\Spec R\to \Gr_{G,\Divtil, \mu}$ such that for all the geometric points $\Spec C'\to \Spec R$, the corresponding $C'$-valued point is in 
\[
G(B^+_{\xi_1,\xi_2}(C'))\mu(\xi_2)G(B^+_{\xi_1,\xi_2}(C'))/G(B^+_{\xi_1,\xi_2}(C')).
\]
\end{enumerate}
\end{defi}
We sometimes write $\Gr_{G, \leq \mu}$ or $\Gr_{\leq \mu}$ for $\Gr_{G,\Divtil, \leq \mu}$, and similarly for $\Gr_{G,\Divtil, \mu}$.
Define the section
\[
[\mu]\colon \Divtil\to \Gr_{G,\Divtil}
\]
by 
\begin{align*}
\begin{array}{cccl}
\Divtil&\to &L\bb{G}_m&\overset{\mu}{\to} LT\subset LG\to \Gr_{G,\Divtil}\\
(\xi_1,\xi_2)&\mapsto &\xi_2.&
\end{array}
\end{align*}
\subsection{Affine flag variety over $\Divtil$}
We can also define a version of the affine flag variety over $\Divtil$.
Define
\[
\cal{I}\colon \rm{PerfAffSch}^{\rm{op}}_{/\Divtil}\to \rm{Sets}
\]
by 
\[
\Spec R\mapsto \left\{g\in G(B^+_{\xi_1,\xi_2}(R))\relmiddle| \begin{array}{l}
\text{The image of $g$ in $G(B^+_{\xi_1,\xi_2}(R)/\xi_2)$}\\
\text{is in $B(B^+_{\xi_1,\xi_2}(R)/\xi_2)$.}
\end{array} \right\}
\]
\begin{defi}
The affine flag variety over $\Divtil$ is defined as the following quotient stack in \'{e}tale topology:
\[
\Fl_{G,\Divtil}:=[L_{\Divtil}G/\cal{I}].
\]
\end{defi}
\begin{rmk}
Unlike the case of a usual affine flag variety, there does not exist a subgroup $I \subset G$ such that $\cal{I}=L^+I$.
\end{rmk}
\begin{prop}\label{prop:mapFlGr}
There is a natural map
\[
\Fl_{G,\Divtil}\to \Gr_{G,\Divtil}, 
\]
which is a fiber bundle with fibers $(\bar{G}/\bar{B})^{\perf}$.
In particular, it is perfectly proper and perfectly smooth.
\end{prop}
\begin{proof}
This follows from the isomorphism
\[
L^+G/\cal{I}\cong G/B,
\]
which is the consequence of Lemma \ref{lemm:BquotR}.
\end{proof}
Assume $G$ is split.
Just as in the case of affine Grassmannian, there is a decomposition 
\[
G(B_{\xi_1,\xi_2}(C'))=\coprod_{w\in \wt{W}} \cal{I}(C')w \cal{I}(C')
\]
for any algebraically closed field $C'$ containing $k_E$, where $\wt{W}$ is an affine Weyl group
\[
\wt{W}=N_G(T)(B_{\xi_1,\xi_2}(C'))/T(B_{\xi_1,\xi_2}^+(C')), 
\]
which does not depend on the choice of $C'$.

In the same way as above, we can define subfunctors
\begin{align*}
&\Fl_{G,\Divtil,\leq w}\subset \Fl_{G,\Divtil}\text{, and}\\
&\Fl_{G,\Divtil,w}\subset \Fl_{G,\Divtil}
\end{align*}
called the affine Schubert variety and affine Schubert cell, respectively.
\section{Affine Grassmannian over $\Divtil$ is an ind-perfectly projective ind-scheme}\label{scn:indsch}
In this subsection, we show that $\Gr_{G,\Divtil}$ is an ind-perfectly projective ind-scheme over $\Divtil$.
The argument here is almost the same as the proof of \cite[Theorem 8.3, Corollary 9.6]{BS}.
\subsection{Preliminary on line bundles}\label{ssc:prelinbdl}
First, we define several relative properties for line bundles on perfect schemes.
In this subsection, let $k$ be a perfect field of characteristic $p$.
The following are the definitions in \cite{CT}:
\begin{defi}(\cite{CT})
Let $f\colon X\to S$ be a proper morphism of noetherian schemes.
Let $\cal{L}$ be a line bundle on $X$.
\begin{enumerate}
\item $\cal{L}$ is $f$-nef if for any field $K$ and morphism $\Spec K\to S$, the pullback $\cal{L}|_{X\times_S \Spec K}$ is nef.
\item $\cal{L}$ is $f$-free if the natural homomorphism $f^*f_*\cal{L}\to \cal{L}$ is surjective.
\item $\cal{L}$ is {$f$-very ample} if $\cal{L}$ is $f$-free and the induced map $X\to \bb{P}(f_*\cal{L})$ is a closed immersion.
\item $\cal{L}$ is {$f$-ample} if $\cal{L}^{\otimes m}$ is $f$-very ample for some $m\geq 1$.
\item $\cal{L}$ is {$f$-semi-ample} if $\cal{L}^{\otimes m}$ is $f$-free for some $m\geq 1$.
\item $\cal{L}$ is {$f$-weakly big} if there exists an $f$-ample line bundle $\cal{A}$ on $X$ and a positive integer $m\geq 1$ such that 
\[
(f|_{X_{\rm{red}}})_*((\cal{L}^{\otimes m}\otimes \cal{A})|_{X_{\rm{red}}})\neq 0.
\]
\end{enumerate}
\end{defi}
For a morphism of the perfections of noetherian schemes, the above properties are defined by passing to some models.
Namely, recall the result on the existence of models:
\begin{prop}
Let $X$ and $S$ be perfectly finitely presented schemes over $k$ (i.e. perfect qcqs schemes that can be written as the perfection of a finitely presented scheme over $k$).
Let $f\colon X\to S$ be a morphism.
Let $\cal{L}$ be a line bundle on $X$.
Then there exists a pair 
\[
(f'\colon X'\to S',\cal{L}'),
\]
such that $f'\colon X'\to S'$ is a morphism between finitely presented schemes satisfying $(f')^{\perf}=f$, and that $\cal{L}'$ is a line bundle on $X'$ satisfying $(\cal{L}')^{\perf}=\cal{L}$.
Here $(\cal{L}')^{\perf}$ is the pullback of $\cal{L}'$ under the canonical morphism $X=(X')^{\perf} \to X'$.
\end{prop}
\begin{proof}
\cite[Proposition A.17, Lemma A.22]{Zhumixed}.
\end{proof}
The pair $(f'\colon X'\to S',\cal{L}')$ is called a (finitely presented) model of $(f\colon X\to S,\cal{L})$.
\begin{lemm}\label{lemm:perfsemiample}
Let $f\colon X\to S$ be a morphism of perfectly finitely presented schemes over $k$.
Let $(f'\colon X'\to S',\cal{L}'), (f''\colon X''\to S'',\cal{L}'')$ be a finitely presented model of $(f\colon X\to S,\cal{L})$.
The following are equivalent:
\begin{enumerate}
\item $\cal{L}'$ is $f'$-semi-ample.
\item $\cal{L}''$ is $f''$-semi-ample.
\end{enumerate}
We say that $\cal{L}$ is $f$-semi-ample if the above conditions are satisfied.
\end{lemm}
\begin{proof}
Since $f'$ and $f''$ are finitely presented, there exists a commutative diagram
\[
\xymatrix{
(X')^{(m)}\ar[r]^{f'}\ar[d]_{g}&(S')^{(m)}\ar[d]^{h}\\
X''\ar[r]_{f''}&S''\\
}
\]
such that $g^{\perf}=\rm{id}_X$, $h^{\perf}=\rm{id}_S$ and $g^*\cal{L}''=(\cal{L}')^{\otimes p^m}$.
If $\cal{L}''$ is $f''$-semi-ample, then
\[
(f'')^*\cal{Q}''\to (\cal{L}'')^{\otimes N}
\]
is surjective for some $N$ and some quasi-coherent sheaf $\cal{Q}''$ on $S''$.
In fact, 
\[
(f'')^*(f'')_*(\cal{L}'')^{\otimes N}\to (\cal{L}'')^{\otimes N}
\]
for some $N$, so we can put $\cal{Q}''=(f'')_*(\cal{L}'')^{\otimes N}$.
Thus $g^*(f'')^*\cal{Q}''\to g^*(\cal{L}'')^{\otimes N}$ is surjective, that is, 
\[
(f')^{*}h^*\cal{Q}''\to (\cal{L}')^{\otimes Np^m}
\]
 is surjective.
This factors through
\[
(f')^*(f')_*(\cal{L}')^{\otimes Np^m}\to (\cal{L}')^{\otimes Np^m}, 
\]
which is therefore surjective.
Thus $\cal{L}'$ is $f'$-semi-ample.

By the same argument, if $\cal{L}'$ is $f'$-semi-ample, then $\cal{L}''$ is $f''$-semi-ample.
\end{proof}
\begin{lemm}\label{lemm:perfectpush}
Let $f\colon X\to S$ be a morphism of perfectly finitely presented schemes over $k$.
Let $(f'\colon X'\to S',\cal{L}')$ be a finitely presented model of $(f\colon X\to S,\cal{L})$.
Then there is a canonical isomorphism
\[
f_*\cal{L}=\lim_{\substack{\longrightarrow \\ m\geq 1}} f'_*\cal{L}'^{\otimes p^m}
\]
as sheaves of abelian groups on $|S|=|S'|$.
Here the transition map
\[
f'_*\cal{L}'^{\otimes p^m}\to f'_*\cal{L}'^{\otimes p^{m+1}}
\]
is induced by a map $a\mapsto a^p$.
\end{lemm}
\begin{proof}
Let $|f|\colon |X|\to |S|$ be the underlying map of topological spaces of $f$.
It suffices to show that for any open subset $U\subset |S|=|S'|$, 
\[
\Gamma(|f|^{-1}(U),\cal{L})=\lim_{\substack{\longrightarrow \\ m}} \Gamma(|f|^{-1}(U), \cal{L}'^{\otimes p^m}).
\]
This equality can be easily deduced from the fact that $\cal{L}=\cal{L}'\otimes_{\cal{O}_{X'}}\cal{O}_X$ as a sheaf of abelian group on $|X|$.
See also the proof of \cite[Lemma 3.6]{BS}.
\end{proof}
Recall the following lemma in \cite{BS}:
\begin{lemm}(\cite[Lemma 3.6]{BS})\label{lemm:absample}
Let $X$ be a perfectly finitely presented scheme over $k$ and $\cal{L}$ a line bundle on $X$.
Let $(X',\cal{L}')$ be a finitely presented model of $(X,\cal{L})$.
The following are equivalent:
\begin{enumerate}
\item $\cal{L}'$ is ample, i.e. for any $x\in X'$, there exists a section $s'\in \Gamma(X',\cal{L}^{\prime \otimes m})$ for some $m$ such that
\[
X'_{s'}:=\{y\in X'\mid s'(y)\neq 0\}
\] 
is an affine neighborhood of $x$.
\item $\cal{L}$ is ample, in the same sense as (i).
\end{enumerate}
\end{lemm}
\begin{lemm}\label{lemm:perfample}
Let $f\colon X\to S$ be a morphism of perfectly finitely presented schemes over $k$, and $\cal{L}$ a line bundle on $X$.
Let $(f'\colon X'\to S',\cal{L}')$ be a finitely presented model of $(f\colon X\to S,\cal{L})$.
The following are equivalent:
\begin{enumerate}
\item $\cal{L}'$ is $f'$-ample.
\item For any $x'\in X'$, there exists an affine open neighborhood $V'$ of $f'(x')$ in $S'$ such that $\cal{L}'|_{f^{\prime -1}(V')}$ is ample in the sense of Lemma \ref{lemm:absample}.
\item For any $x\in X$, there exists an affine open neighborhood $V$ of $f(x)$ in $S$ such that $\cal{L}|_{f^{-1}(V)}$ is ample in the sense of Lemma \ref{lemm:absample}.
\end{enumerate}
We say that $\cal{L}$ is $f$-ample if the above conditions are satisfied.
\end{lemm}
\begin{proof}
For the equivalence of (i) and (ii), see \cite[Tag 01VU]{Sta}.
The equivalence of (ii) and (iii) follows from Lemma \ref{lemm:absample}.
\end{proof}
\begin{lemm}\label{lemm:perfwb}
Let $f\colon X\to S$ be a morphism of perfectly finitely presented schemes over $k$, and $\cal{L}$ a line bundle on $X$.
Let $(f'\colon X'\to S',\cal{L}')$ be a finitely presented model of $(f\colon X\to S,\cal{L})$.
The following are equivalent:
\begin{enumerate}
\item\label{lemm:cond:wb} $\cal{L}'$ is $f'$-weakly big.
\item\label{lemm:cond:perfwb} There exists an $f$-ample sheaf $\cal{A}$ and a positive integer $m$ such that 
\[
f_*(\cal{L}^{\otimes m}\otimes \cal{A})\neq 0.
\]
\end{enumerate}
We say that $\cal{L}$ is $f$-weakly big if the above conditions are satisfied.
\end{lemm}
\begin{proof}
Assume (i).
Then there exists an $f'$-ample sheaf $\cal{A}'$ and $m\geq 1$ such that 
\[
(f^{\prime}|_{X'_{\rm{red}}})_*(((\cal{L}')^{\otimes m}\otimes \cal{A}')|_{X'_{\rm{red}}})\neq 0.
\]
Let $\cal{A}$ be a pullback of $\cal{A}'$ to $X$. 
By Lemma \ref{lemm:perfample}, $\cal{A}$ is $f$-ample.
Moreover, by Lemma \ref{lemm:perfectpush}, it holds that
\begin{align}\label{eqn:perfpushwb}
f_*(\cal{L}^{\otimes m}\otimes \cal{A})=\lim_{\substack{\longrightarrow\\ e\geq 1}} (f'|_{X'_{\rm{red}}})_*((\cal{L}')^{\otimes m}\otimes \cal{A}')^{\otimes p^e}|_{X'_{\rm{red}}}.
\end{align}
For any locally free $\cal{O}_{X'_{\rm{red}}}$-module $\cal{F}$, the homomorphism
\[
\cal{F}\to \cal{F}^{\otimes p},\ a\mapsto a^p
\]
is injective by the reducedness.
Thus
\[
(f'|_{X'_{\rm{red}}})_*\cal{F}\to (f'|_{X'_{\rm{red}}})_*\cal{F}^{\otimes p}
\]
is also injective.
Therefore, the transition homomorphisms in the right-hand side of (\ref{eqn:perfpushwb}) are injective.
In particular, there is an injection
\[
(f'|_{X'_{\rm{red}}})_*((\cal{L}')^{\otimes m}\otimes \cal{A}')|_{X'_{\rm{red}}} \to f_*(\cal{L}^{\otimes m}\otimes \cal{A}).
\]
We get $f_*(\cal{L}^{\otimes m}\otimes \cal{A})\neq 0$.

Conversely, assume (ii).
There exists an $f$-ample sheaf $\cal{A}$ and a positive integer $m$ such that 
\[
f_*(\cal{L}^{\otimes m}\otimes \cal{A})\neq 0.
\]
By Lemma \ref{lemm:perfample}, there exists an $f'$-ample sheaf $\cal{A}'$ whose pullback to $X$ is $\cal{A}$.
Then the equation (\ref{eqn:perfpushwb}) holds, and thus
\[
(f'|_{X'_{\rm{red}}})_*((\cal{L}')^{\otimes m}\otimes \cal{A}')^{\otimes p^e}|_{X'_{\rm{red}}}\neq 0
\]
for some $e\geq 1$.
Since $(\cal{A}')^{^{\otimes p^e}}$ is $f'$-ample, we get (i).
\end{proof}
\begin{lemm}\label{lemm:perfnef}
Let $f\colon X\to S$ be a morphism of perfectly finitely presented schemes over $k$, and $\cal{L}$ a line bundle on $X$.
Let $(f'\colon X'\to S',\cal{L}')$ be a finitely presented model of $(f\colon X\to S,\cal{L})$.
The following are equivalent:
\begin{enumerate}
\item $\cal{L}'$ is $f'$-nef.
\item For each closed point $\Spec k'\to S'$, the pullback $\cal{L}'|_{X'\times_{S'}\Spec k'}$ is nef.
\item For each closed point $\Spec k'\to S$, the pullback $\cal{L}|_{X\times_{S}\Spec k'}$ satisfies the following:\\
For any smooth projective curve $C'$ over $k'$ and any non-constant morphism $C^{\prime\perf}\to X\times_{S}\Spec k'$, the pullback $\cal{L}|_{C^{\prime\perf}}$ of $\cal{L}$ to $C^{\prime\perf}$ is ample (in the sense of Lemma \ref{lemm:absample}).
\end{enumerate}
We say that $\cal{L}$ is $f$-nef if the above conditions are satisfied.
\end{lemm}
\begin{proof}
For the equivalence of (i) and (ii), see \cite[Lemma 2.6]{CT}.

For the equivalence of (ii) and (iii), we may assume $S=\Spec k$ since $|S|=|S'|$. 
Assume (ii) and take any smooth projective curve $C'$ over $k$ and any non-constant morphism $g\colon C^{\prime\perf}\to X$.
There exists a model of $g$ of the form
\[
C^{\prime(m)}\to X'
\]
for some $m\geq 0$.
By (ii), the line bundle $\cal{L}'|_{C^{\prime(m)}}$ is ample.
By construction, $\cal{L}|_{C^{\prime\perf}}$ is the pullback of $\cal{L}'|_{C^{\prime(m)}}$ to $C^{\prime\perf}$.
Thus $\cal{L}|_{C^{\prime\perf}}$ is ample by Lemma \ref{lemm:absample}.

Conversely, assume (iii) and take any smooth projective curve $C'$ over $k'$ and any non-constant morphism $g'\colon C'\to X'$.
Consider the perfection $g\colon C\to X$.
By (iii), the line bundle $\cal{L}|_{C^{\prime\perf}}$ is ample.
By construction, $\cal{L}|_{C^{\prime\perf}}$ is the pullback of $\cal{L}'|_{C^{\prime}}$ to $C^{\prime\perf}$.
Thus $\cal{L}'|_{C^{\prime}}$ is ample by Lemma \ref{lemm:absample}.
\end{proof}
\begin{defi}[Exceptional locus]
\begin{enumerate}
\item(\cite[\S 2.1.1]{CT}) Let $f'\colon X'\to S'$ be a morphism of finitely presented schemes over $k$, and $\cal{L}'$ an $f'$-nef line bundle on $X'$.
The exceptional locus $\bb{E}_{f'}(\cal{L}')$ is defined by the union of all the reduced closed subset $V'\subset X'$ such that $\cal{L}'|_{V'}$ is not $f'|_{V'}$-weakly big.
\item Let $f\colon X\to S$ be a morphism of perfectly finitely presented schemes over $k$, and $\cal{L}$ a line bundle on $X$.
Assume $\cal{L}$ is $f$-nef (in the sense of Lemma \ref{lemm:perfnef}).
The exceptional locus $\bb{E}_f(\cal{L})$ is defined by the union of all the perfect closed subset $V\subset X$ such that $\cal{L}|_{V}$ is not $f|_{V}$-weakly big (in the sense of Lemma \ref{lemm:perfwb}).
\end{enumerate}
\end{defi}
\begin{cor}\label{cor:perfexloc}
Let $f\colon X\to S$ be a morphism of perfectly finitely presented schemes over $k$, and $\cal{L}$ an $f$-nef line bundle on $X$.
Let $(f'\colon X'\to S',\cal{L}')$ be a finitely presented model of $(f\colon X\to S,\cal{L})$.
(The line bundle $\cal{L}'$ is $f'$-nef by Lemma \ref{lemm:perfnef}).
Then $\bb{E}_{f}(\cal{L})$ is perfect closed subset and it holds that
\[
\bb{E}_{f}(\cal{L})=(\bb{E}_{f'}(\cal{L}'))^{\perf}
\]
in $X=X^{\prime\perf}$.
\end{cor}
\begin{proof}
This follows from Lemma \ref{lemm:perfwb}.
\end{proof}
\begin{thm}[Perfect relative Keel's theorem]
Let $f\colon X\to S$ be a morphism of perfectly finitely presented schemes over $k$, and $\cal{L}$ an $f$-nef line bundle on $X$.
Then $\cal{L}$ is $f$-semi-ample if and only if $\cal{L}|_{\bb{E}_f(\cal{L})}$ is $f|_{\bb{E}_f(\cal{L})}$-semi-ample.
\end{thm}
\begin{proof}
By choosing a model $(f'\colon X'\to S',\cal{L}')$ of $(f\colon X\to S,\cal{L})$ and using Lemma \ref{lemm:perfnef}, Lemma \ref{lemm:perfsemiample} and Corollary \ref{cor:perfexloc}, the theorem reduced to the (non-perfect) finitely presented case.
Then it follows from \cite[Proposition 2.20]{CT}.
\end{proof}
\subsection{Affine Grassmannain of $\GL_n$}
In this subsection, assume $G=\GL_n$.
Let $T$ be the subgroup of diagonal matrices, and $B$ the subgroup of upper triangular matrices.
Note that there is a canonical identification
\begin{align*}
\bb{X}_*(T)&=\bb{Z}^n\\
\bb{X}_*^+(T)&=\{(m_1,\ldots,m_n)\in \bb{Z}^n\mid m_1\geq \cdots\geq m_n\}.
\end{align*}

In this case, the affine Grassmannian $\Gr_{\GL_n,\Divtil}$ can be written as follows:
\[
\Gr_{\GL_n,\Divtil}(R)=\left\{(\xi_1,\xi_2,\cal{E})\relmiddle|
\begin{array}{l}
(\xi_1,\xi_2)\in \Divtil(R)\\
\text{$\cal{E}$ is a finite projective $B^+_{\xi_1,\xi_2}(R)$-submodule of $B_{\xi_1,\xi_2}(R)^n$}\\
\cal{E}[1/\xi_2]=B_{\xi_1,\xi_2}(R)^n
\end{array}
\right\}.
\]
In other words, the affine Grassmannian is a moduli space of $B^+_{\xi_1,\xi_2}(R)$-lattices in $B_{\xi_1,\xi_2}(R)^n$:
\begin{defi}
We call $\cal{E}$ a $B^+_{\xi_1,\xi_2}(R)$-lattice in $B_{\xi_1,\xi_2}(R)^n$ if $\cal{E}$ is  a finite projective $B^+_{\xi_1,\xi_2}(R)$-submodule of $B_{\xi_1,\xi_2}(R)^n$ such that $\cal{E}[1/\xi_2]=B_{\xi_1,\xi_2}(R)^n$.
\end{defi}
\begin{defi}
Let $k$ be a perfect field over $k_E$, and fix $(\xi_1,\xi_2)\in \Divtil(k)$.
Let $E_1,E_2$ be $B^+_{\xi_1,\xi_2}(k)$-lattices in $B_{\xi_1,\xi_2}(k)^n$.
Since $B^+_{\xi_1,\xi_2}(k)$ is a discrete valuation field by Lemma \ref{lemm:Bdvr}, there exists a basis $(e_1,\ldots,e_n)$ of $E_1$ and a basis $(f_1,\ldots,f_n)$ of $E_2$ such that
\[
e_i=\xi_2^{m_i}f_i
\]
for some $\mu=(m_1,\ldots,m_n)\in \bb{X}_*^+(T)$.
Moreover, $\mu$ does not depend on the choice of the basis.
This $\mu\in \bb{X}_*^+(T)$ is called {the relative position of $E_1$ to $E_2$}.
\end{defi}
\begin{defi}
Let $R$ be a perfect ring over $k_E$, and fix $(\xi_1,\xi_2)\in \Divtil(R)$.
Let $\cal{E}_1,\cal{E}_2$ be finite projective $B^+_{\xi_1,\xi_2}(R)$-lattices in $B_{\xi_1,\xi_2}(R)^n$.
We say that {the relative position of $\cal{E}_2$ to $\cal{E}_1$ is $\mu$ (resp. $\leq \mu$)} if the relative position of $\cal{E}_2\otimes_{B^+_{\xi_1,\xi_2}(R)}B^+_{\xi_1,\xi_2}(k(x))$ to $\cal{E}_1\otimes_{B^+_{\xi_1,\xi_2}(R)}B^+_{\xi_1,\xi_2}(k(x))$ is $\mu$ (resp. $\leq \mu$) for each point $x\in \Spec R$.
\end{defi}
For $R\in \Perf_{k_E}$ and $(\xi_1,\xi_2)\in \Divtil(R)$, put
\[
\cal{E}_0:=\cal{E}_{0,\GL_n,\xi_1,\xi_2}:=B^+_{\xi_1,\xi_2}(R)^n.
\]
By the definition of $\Gr_{\leq \mu}=\Gr_{\GL_n,\Divtil,\leq \mu}$ and $\Gr_{\mu}=\Gr_{\GL_n,\Divtil,\mu}$, one can check that
\begin{align*}
\Gr_{\leq \mu}(R)=\{&(\xi_1,\xi_2,\cal{E})\in \Gr_{\GL_n,\Divtil}(R)\mid \\
& \text{The relative position of $\cal{E}$ to $\cal{E}_0$ is $\leq \mu$}\},\\
\Gr_{\mu}(R)=\{&(\xi_1,\xi_2,\cal{E})\in \Gr_{\GL_n,\Divtil}(R)\mid \\
& \text{The relative position of $\cal{E}$ to $\cal{E}_0$ is $\mu$}\}.
\end{align*}
Note that there is a canonical isomorphism
\begin{align*}
\Gr_{\GL_n,\Divtil,\leq (m_1,\ldots,m_n)}&\overset{\sim}{\to} \Gr_{\GL_n,\Divtil,\leq (m_1+l,\ldots,m_n+l)},\\
\cal{E}&\mapsto \xi_2^l\cal{E}.
\end{align*}
Thus we only have to consider the case $m_n\geq 0$.
In this case, there is an inclusion between the lattices:
\begin{lemm}(cf. \cite[Lemma 1.5]{Zhumixed}) \label{lemm:condincl}
Let $R$ be a perfect ring over $k_E$, and fix $(\xi_1,\xi_2)\in \Divtil(R)$.
Let $\cal{E}_1,\cal{E}_2$ be $B^+_{\xi_1,\xi_2}(R)$-lattices in $B_{\xi_1,\xi_2}(R)^n$.
If the relative position of $\cal{E}_1$ to $\cal{E}_2$ is $\leq \mu$ for some $\mu=(m_1,\ldots,m_n)$ with $m_n\geq 0$, then there is an inclusion $\cal{E}_1\subset \cal{E}_2$.
\end{lemm}
\begin{proof}
The proof is the same as the proof in \cite[Lemma 1.5]{Zhumixed}, using that $a\in B_{\xi_1,\xi_2}(R)$ is an element of $B^+_{\xi_1,\xi_2}(R)$ if and only if the image of $a$ in $B_{\xi_1,\xi_2}(k(x))$ is in $B^+_{\xi_1,\xi_2}(k(x))$ for every $x\in \Spec R$.
\end{proof}
If $\cal{E}_1\subset \cal{E}_2$, then the notion of the relative position of $\cal{E}_2$ to $\cal{E}_1$ can be rephrased as the notion for the quotient $\cal{E}_2/\cal{E}_1$.
\begin{defi}
Let $k$ be a perfect field over $k_E$, and fix $(\xi_1,\xi_2)\in \Divtil(k)$.
Let $Q$ be a finitely generated $\xi_2$-power torsion $B^+_{(\xi_1,\xi_2)}(k)$-module.
For $\mu=(m_1,\ldots,m_n)\in \bb{Z}^n$ satisfying $m_1\geq \cdots\geq m_n\geq 0$, we say that {$Q$ is of type $\mu$} if
\[
Q\cong \bigoplus_{i=1}^n B^+_{(\xi_1,\xi_2)}(k)/\xi_2^{m_i}.
\]
(If $Q$ can be generated by $n$ elements, such $\mu$ always exists by Lemma \ref{lemm:Bdvr}.)
\end{defi}
For a perfect scheme $X$ over $\Divtil$, we say that a sheaf $\cal{Q}$ on $X$ is a $\cal{B}^+_{\xi_1,\xi_2}(\cal{O}_X)$-module if for any open affine subset $\Spec R\subset X$, the set $\cal{Q}(R)$ has a structure of a $B^+_{\xi_1,\xi_2}(R)$-module compatibly for $R$, where $(\xi_1,\xi_2)$ is a pair corresponding to the composition $\Spec R\subset X\to \Divtil$.
\begin{defi}
Let $X$ be a perfect scheme over $\Divtil$.
Let $\cal{Q}$ be a finitely generated $\xi_2$-power torsion quasi-coherent $B^+_{\xi_1,\xi_2}(\cal{O}_X)$-module.
For $\mu=(m_1,\ldots,m_n)\in \bb{Z}^n$ satisfying $m_1\geq \cdots\geq m_n\geq 0$, we say that {$\cal{Q}$ is of type $\mu$ (resp. $\leq \mu$)} if
\[
\cal{Q}\otimes_{B^+_{\xi_1,\xi_2}(\cal{O}_X)} B^+_{\xi_1,\xi_2}(k(x))
\] 
is of type $\mu$ (resp. $\leq \mu$) for all $x\in X$.
\end{defi}
By Lemma \ref{lemm:condincl}, the following holds:
\begin{cor}
Let $\cal{E}_1,\cal{E}_2$ be $B^+_{\xi_1,\xi_2}(R)$-lattices in $B_{\xi_1,\xi_2}(R)^n$.
Let $\mu=(m_1,\ldots,m_n)\in \bb{Z}^n$ be a sequence satisfying $m_1\geq \cdots\geq m_n\geq 0$.
\begin{enumerate}
\item The relative position of $\cal{E}_2$ to $\cal{E}_1$ is $\leq \mu$ if and only if $\cal{E}_2\subset \cal{E}_1$ and $\cal{E}_1/\cal{E}_2$ is of type $\leq \mu$.
\item The relative position of $\cal{E}_2$ to $\cal{E}_1$ is $\mu$ if and only if $\cal{E}_2\subset \cal{E}_1$ and $\cal{E}_1/\cal{E}_2$ is of type $\mu$.
\end{enumerate}
\end{cor}
Thus
\begin{align*}
\Gr_{\leq \mu}(R)=\{&(\xi_1,\xi_2,\cal{E})\in \Gr(R)\mid \text{$\cal{E}\subset \cal{E}_0$ and $\cal{E}_0/\cal{E}$ is of type $\leq \mu$}\},\\
\Gr_{\mu}(R)=\{&(\xi_1,\xi_2,\cal{E})\in \Gr(R)\mid \text{$\cal{E}\subset \cal{E}_0$ and $\cal{E}_0/\cal{E}$ is of type $\mu$}\}.
\end{align*}
Let $\mu=(m_1,\ldots,m_n)\in \bb{X}_*^+(T)$ be a dominant cocharacter satisfying $m_n\geq 0$.
For $1\leq i\leq n$, put 
\[
\omega_i=(\overbrace{1,\ldots,1}^{i},0,\ldots,0)\in \bb{X}_*^+(T).
\]
There is a decomposition
\begin{align}\label{eqn:minudecomp}
\mu=\mu_1+\cdots+\mu_{m_1}
\end{align}
such that $\mu_j= \omega_{n_{\mu}(j)}$ for some $n_{\mu}(j)$.
For simplicity, we assume that if $j<j'$, then $n_{\mu}(j)\leq n_{\mu}(j')$.
Under this assumption, the above decomposition is unique.
\begin{defi}\label{defi:DemGr}
{The Demazure resolution} $\wt{\Gr}_{\mu}=\wt{\Gr}_{\GL_n,\Divtil,\leq \mu}$ of $\Gr_{\GL_n,\Divtil,\leq \mu}$ is the functor on $\rm{PerfAffSch}^{\rm{op}}_{/\Divtil}$ defined by
\begin{align*}
\Spec R\mapsto \{&(\cal{E}_{m_1}\subset \cdots \subset \cal{E}_1\subset \cal{E}_0)\mid \\
& \left.\begin{array}{l}\text{$\cal{E}_j$'s are $B^+_{\xi_1,\xi_2}(R)$-lattices in $B_{\xi_1,\xi_2}(R)^n$}\\
\text{The relative position of $\cal{E}_j$ to $\cal{E}_{j-1}$ is $\mu_j$}
\end{array}
\right\}.
\end{align*}
\end{defi}
There is a map
\begin{align*}
    \psi\colon \wt{\Gr}_{G,\Divtil, \mu}&\to \Gr_{G,\Divtil},\\(\cal{E}_{m_1}\subset \cdots\subset \cal{E}_0)&\mapsto (\cal{E}_{m_1}\subset \cal{E}_0).
\end{align*}
By the same argument as the proof of \cite[Remark 8.5]{BS}, one can show the following theorem:
\begin{thm}\label{thm:psirat}
The map $\psi\colon \wt{\Gr}_{\mu}\to \Gr_{\leq \mu}$ satisfies 
\[
R\psi_*\cal{O}=\cal{O},
\]
in particular, surjective, and 
\[
\psi|_{\psi^{-1}(\Gr_{\mu})}\colon \psi^{-1}(\Gr_{\mu})\to \Gr_{\mu}
\]
is an isomorphism.
\end{thm}
\begin{proof}
The proof is the same as \cite{BS}.

We can define a notion of a Demazure scheme in \cite[Definition 7.10]{BS} for a quasi-coherent $B^+_{\xi_1,\xi_2}(\cal{O}_X)$-module instead of a quasi-coherent $W_{\cal{O}_E}(\cal{O}_X)$-module.
Then we can write $\wt{\Gr}_{\mu}$ as a Demazure scheme and we can show the analogous result to \cite[Proposition 7.11]{BS} and \cite[Lemma 7.13]{BS}, which implies the theorem.
\end{proof}
%
We still write $\Gr_{\mu}$ for $\psi^{-1}(\Gr_{\mu})$.
We can show that $\wt{\Gr}_{\mu}$ is a successive perfect Grassmannian bundle over $\Divtil$:
\begin{prop}(cf.\cite[Proposition 8.6]{BS}) \label{prop:Demsmproj}
The functor $\wt{\Gr}_{\mu}$ is representable by the perfection of a smooth projective scheme over $\Divtil\cong \GL_{2,k_E}^{\perf}$.
\end{prop}
\begin{proof}
The proof is the same as \cite[Proposition 8.6]{BS}.
\end{proof}
\begin{cor}\label{cor:Grclosed}
The subfunctor
\[
\Gr_{\leq \mu}\subset \Gr
\]
is a closed subfunctor.
Moreover, the subfunctor
\[
\Gr_{\mu}\subset \Gr_{\leq \mu}
\]
is an open subfunctor.
\end{cor}
\begin{proof}
By Theorem \ref{thm:psirat}, the Schubert variety $\Gr_{\leq \mu}$ is an image of the proper scheme $\wt{\Gr}_{\mu}$.
This implies the first part.
The second part follows from the first part and that
\[
\Gr_{\mu}=\Gr_{\leq \mu}\setminus \bigcup_{\nu<\mu} \Gr_{\leq \mu}.
\]
\end{proof}
\begin{defi}
For $1\leq j\leq m_1$, let $\cal{Q}_j$ be the universal bundle on $\wt{\Gr}_{\mu}$ corresponding to $\cal{E}_{j-1}/\cal{E}_{j}$, i.e. 
\[
\cal{Q}_j|_{\Spec R}=\cal{E}_{j-1}/\cal{E}_{j}
\]
for any $\Spec R\to \wt{\Gr}_{\mu}$ corresponding to $(\cal{E}_{m_1}\subset \cdots \subset \cal{E}_1\subset \cal{E}_0)$, noting that $\cal{E}_{j-1}/\cal{E}_{j}$ is a finite projective $R$-module by the same argument as \cite[Lemma 7.8]{BS}.

The line bundle $\wt{\cal{L}}_{\mu}$ on $\wt{\Gr}_{\mu}$ is defined as 
\[
\wt{\cal{L}}_{\mu}=\bigotimes_{j=1}^{m_1}\det(\cal{Q}_j).
\]
\end{defi}
We can show that the line bundle $\wt{\cal{L}}_{\mu}$ on $\wt{\Gr}_{\mu}$ descends to a line bundle $\cal{L}_{\mu}$ on $\Gr_{\leq \mu}$:
\begin{thm}(cf.\cite[Theorem 8.8]{BS})\label{thm:Lexists}
There is a unique line bundle $\cal{L}_{\mu}$ on $\Gr_{\leq \mu}$ such that the pullback $\psi^*\cal{L}_{\mu}$ is $\wt{\cal{L}}_{\mu}$.

For $\nu \leq \mu$, the pullback of $\cal{L}_{\mu}$ to $\Gr_{\leq \nu}$ is $\cal{L}_{\nu}$, compatible in $\nu$.
\end{thm}
\begin{proof}
The proof is the same as \cite[Theorem 8.8]{BS}.
%
%
\end{proof}
Now we can show the projectivity of the affine Grassmannian. 
\begin{thm}(cf.\cite[Theorem 8.3]{BS}) \label{thm:Grrep}
The functor $\Gr_{\GL_n,\Divtil,\leq \mu}$ is representable by a perfectly proper perfectly finitely presented scheme over $\Divtil$, and has a natural relatively ample line bundle $\cal{L}_{\mu}$ over $\Divtil$.
\end{thm}
\begin{proof}
The proof is the same as \cite[Theorem 8.3]{BS}.
Here we need the result in \S \ref{ssc:prelinbdl} such as perfect relative Keel's theorem.
\end{proof}
\subsection{Affine Grassmannian for general $G$}
Now we can show representability and projectivity of $\Gr_G$ for general $G$.
\begin{thm}(cf.\cite[Theorem 4.5.1]{BD})\label{thm:GrGrep}
Let $G$ be an arbitrary reductive group over $\cal{O}_E[[t]]$.
The affine Grassmannian $\Gr_G$ can be written as an increasing union of the perfections of projective schemes over $\Divtil$, along closed immersions.
\end{thm}
\begin{proof}
The proof is the same as \cite[Theorem 4.5.1]{BD} or \cite[Corollary 9.6]{BS}.
%
%
%
\end{proof}
Next, we show that Schubert cells are smooth schemes.
We need the following proposition:
\begin{prop}(cf.\cite[Proposition VI.2.4]{FS})\label{prop:localbasis}
Assume $G=\GL_n$.
Let $\mu=(m_1,\ldots,m_n)\in \bb{X}_*^+(T)$ and $(\xi_1,\xi_2,\cal{E})\in \Gr_{\GL_n,\Divtil,\mu}(R)$.
There exists \'{e}tale locally on $R$ a $B^+_{\xi_1,\xi_2}(R)$-basis $g_1,\ldots,g_n$ of $\cal{E}$ of the form
\[
g_i=\xi_2^{m_i}f_i
\]
such that $f_1,\ldots,f_n$ is a $B^+_{\xi_1,\xi_2}(R)$-basis of $\cal{E}_0:=B^+_{\xi_1,\xi_2}(R)^n$.
\end{prop}
\begin{proof}
The proof is similar to the proof of \cite[Proposition VI.2.4]{FS}.
%
\end{proof}
Let 
\[
(L^+G)_{\mu}\subset L^+G
\]
be the closed subgroup stabilizing $[\mu]$.
\begin{prop}(cf.\cite[Proposition VI.2.4]{FS}) \label{prop:LGmu}
The Schubert cell $\Gr_{G, \mu}$ is $L^+G$-orbit of $[\mu]$.
In particular, 
\[
\Gr_{G, \mu}\cong L^+G/(L^+G)_{\mu}.
\]
Put
\[
(L^+G)_{\mu}^{\geq m}:=(L^+G)_{\mu}\cap (L^+G)^{\geq m}.
\]
Then there are natural isomorphisms
\begin{align*}
&(L^+G)_{\mu}/(L^+G)_{\mu}^{\geq 1}\cong \bar{P}_{\mu}^- \times_{\Spec k_E}\Divtil\\
&\subset L^+G/(L^+G)^{\geq 1}\cong \bar{G}\times_{\Spec k_E}\Divtil, 
\end{align*}
and
\begin{align*}
&(L^+G)_{\mu}^{\geq m}/(L^+G)_{\mu}^{\geq m+1}\cong (\rm{Lie}(\bar{G}))_{\mu\leq m} \times_{\Spec k_E}\Divtil)\{m\}\\
&\subset (L^+G)^{\geq m}/(L^+G)^{\geq m+1}\cong (\rm{Lie}(\bar{G}) \times_{\Spec k_E}\Divtil)\{m\},
\end{align*}
where $(\rm{Lie}(\bar{G}))_{\mu\leq m}\subset \rm{Lie}(\bar{G})$ is the subspace on which $\mu$ act via weights $\leq m$, and $\bar{P}_{\mu}^-$ is the parabolic subgroup with Lie algebra $(\rm{Lie}(\bar{G}))_{\mu\leq 0}$.
In particular, the Schubert cell $\Gr_{G, \mu}$ is a perfectly smooth scheme over $\Divtil$ of dimension $\langle 2\rho,\mu\rangle $, where $\rho$ is a half sum of positive roots.
\end{prop}
\begin{proof}
The proof is the same as \cite[Proposition VI.2.4]{FS}.
The claim reduces to the case $G=\GL_n$.
In this case, the claim follows from Proposition \ref{prop:localbasis} and several calculations together with Proposition \ref{thm:GrGrep}.
\end{proof}
\subsection{Affine flag variety}
Recall that $\wt{W}$ acts on $\bb{X}^*(T)$.
Let $\fk{a}\subset \bb{X}^*(T)\otimes \bb{Q}$ be an alcove corresponding to the Iwahori group $\cal{I}(B^+_{\xi_1,\xi_2}(C'))\subset G(B_{\xi_1,\xi_2}(C'))$, which does not depend on the choice of an algebraically closed field $C'$.
Let $\{s_i\}$ be the affine simple reflections along the faces of the alcove, generating a normal subgroup $W_{\rm{aff}}\subset \wt{W}$.
Let $\Omega\subset \wt{W}$ be the stabilizer of the alcove.
Then we have
\[
\wt{W}=W_{\rm{aff}}\rtimes \Omega.
\]

For an affine simple reflection $s_i$ in $W_{\rm{aff}}$, we can define a subgroup
\[
\cal{P}_i\subset L^+G
\]
in the same way as \cite[\S VI.5]{FS}.

For $w\in \wt{W}$, we can write
\[
w=s_{i_1}\ldots s_{i_l}\omega,
\]
where $s_{i_1},\ldots,s_{i_l}$ are affine simple reflections in $W_{\rm{aff}}$ and $\omega\in \Omega$.
Write $\dot{w}$ for the element $w$ with such a choice of decomposition.
\begin{defi}
The Demazure variety for $\dot{w}$ is defined by
\[
\rm{Dem}_{G,\Divtil,\dot{w}}=\cal{P}_{i_1}\times^{\cal{I}}\cdots \times^{\cal{I}} \cal{P}_{i_l}\to \Divtil,
\]
and the Demazure resolution for $\dot{w}$ is a map
\[
\rm{Dem}_{G,\Divtil,\dot{w}}\to \Fl_{G,\Divtil}
\]
defined by $(p_1,\ldots,p_l)\mapsto p_1\ldots p_l\omega$.
\end{defi}
\begin{lemm}
The map $\rm{Dem}_{G,\Divtil,\dot{w}}\to \Divtil$ is a successive perfect $\bb{P}^1$-bundle, in particular, perfectly proper and perfectly smooth.
The map
\[
\Psi\colon \rm{Dem}_{G,\Divtil,\dot{w}}\to \Fl_{G,\Divtil}
\]
is perfectly proper and its image is $\Fl_{G,\Divtil,\leq w}$.
Moreover,
\[
\Psi|_{\Psi^{-1}(\Fl_{w})}\colon \Psi^{-1}(\Fl_{w}) \to \Fl_{w}
\]
is an isomorphism.
\end{lemm}
\begin{proof}
The first statement follows from an isomorphism $\cal{P}_{i_j}/\cal{I}\cong \bb{P}^{1,\perf}$.
In particular, $\Psi$ is perfectly proper.
The remaining parts can be checked on the geometric fibers by \cite[Corollary 6.10]{BS}.
Then the lemma follows from the standard argument (cf. \cite[Theorem VI.5.5]{FS}).
\end{proof}
By this lemma and the same argument as Corollary \ref{cor:Grclosed}, we can show the following:
\begin{cor}\label{cor:Flclosed}
The subfunctor $\Fl_{G,\Divtil,\leq w}\subset \Fl_{G,\Divtil}$ is a closed subfunctor.
Moreover, $\Fl_{G,\Divtil,w}\subset \Fl_{G,\Divtil,\leq w}$ is an open subfunctor.
\end{cor}
From the result for affine Grassmannian, we can deduce a result for affine flag variety.
\begin{prop}
Assume $G$ is split.
The Schubert variety $\Fl_{G,\Divtil, \leq w}$ is represented by the perfection of a projective algebraic space over $\Divtil$.
\end{prop}
\begin{proof}
By Proposition \ref{prop:mapFlGr}, there exists a $G/B$-bundle 
\[
f\colon \Fl_{G,\Divtil}\to \Gr_{G,\Divtil}
\]
which is \'{e}tale locally trivial.
Note that $\cal{I}$-orbit $\Fl_{G,\Divtil,w}$ is contained in the inverse image of a $L^+G$-orbit $\Gr_{\mu}$ for some $\mu$.
Thus $\Fl_{G,\Divtil,\leq w}$ is contained in $f^{-1}(\Gr_{\leq \mu})$ for some $\mu$ as a closed subfunctor by Corollary \ref{cor:Flclosed}.
The proposition follows from the fact that $\Gr_{G,\Divtil,\leq \mu}$ is the perfection of a projective scheme over $\Divtil$.
\end{proof}
\begin{cor}
The affine flag variety $\Fl_{G,\Divtil}$ can be represented by the inductive limit of the perfection of projective algebraic spaces over $\Divtil$, along closed immersions.
\end{cor}
\section{Derived category for affine Grassmannian and affine flag variety}
\subsection{Derived category for affine Grassmannian}
For a scheme $S\to \Divtil$, we write
\[
\Gr_{G,S}:=\Gr_{G,\Divtil}\times_{\Divtil} S,
\]
and similarly for Schubert varieties, Schubert cells, affine flag varieties, and so on.
By \S \ref{scn:indsch}, the affine Grassmannian $\Gr_{G, S}$ can be written as an inductive limit of the perfections of projective varieties along closed immersions.
Write 
\[
\Gr_{G,S}=\lim_{\substack{\longrightarrow \\ i\in I}} X_i
\]
as an inductive limit with $X_i$ being $L^+_SG$-stable.
We define
\[
D^b_c(\Gr_{G,S},\Lambda)^{\rm{bd}}
\]
as the inductive limit of $D^b_c(X_i,\Lambda)$'s, where $D^b_c(X_i,\Lambda)$ is a category of complexes on $X$ which becomes locally constant with perfect fibers over a constructible stratification.
The action of $L^+_SG$ on $X_i$ factors through $L^{(m)}_SG:=L^+_SG/(L^+_SG)^{\geq m}$ for some $m\geq 1$.
Put
\[
D_{L^+_SG}(X_i,\Lambda):=D_{L^{(m)}_SG}(X_i,\Lambda)
\]
and define 
\[
D(\Hck_{G,S},\Lambda)^{\rm{bd}}:=D_{L^+_SG}(\Gr_{G,S},\Lambda)^{\rm{bd}}
\]
as the inductive limit of $D_{L^+_SG}(X_i,\Lambda)$'s.
Here $\Hck_{G,S}$ means the quotient stack
\[
\Hck_{G,S}=[L^+_SG\backslash \Gr_{G,S}].
\]
\subsection{Convolution products}
Let $S\to \Divtil$ be a perfect scheme over $\Divtil$.
Consider the diagram
\[
\xymatrix{
[L^+_SG \backslash L_SG\times^{L^+_SG}\Gr_{G,S}]\ar[r]^-{m}\ar[d]_{p}&\Hck_{G,S}\\
\Hck_{G,S}\times \Hck_{G,S}, 
}
\]
where $p$ is a map induced from the identity on $L_SG\times L_SG$, and $m$ is induced from the multiplication map $L_SG\times L_SG\to L_SG$.
For $A,B\in D(\Hck_{G,S},\Lambda)^{\bd}$, we can define the convolution product $A\star B\in D(\Hck_{G,S},\Lambda)^{\bd}$ of $A$ and $B$ by 
\begin{align*}
A\star B=Rm_!p^*(A\boxtimes B).
\end{align*}
We can prove the following results by a standard argument:
\begin{lemm}\label{lemm:convass}
There exists a canonical equivalence
\[
(A\star B)\star C\cong A\star (B\star C)
\]
which makes $D(\Hck_{G,S},\Lambda)^{\bd}$ monoidal triangulated category.
\end{lemm}
\begin{lemm}\label{lemm:convpullback}
Let $S'\to S\to \Divtil$ be morphisms of perfect schemes.
Let
\[
f^*\colon D(\Hck_{G,S},\Lambda)^{\bd}\to D(\Hck_{G,S'},\Lambda)^{\bd}
\]
be a pullback functor.
Then there exists a canonical isomorphism
\[
f^*(A\star B)\cong f^*A\star f^*B.
\]
\end{lemm}
\subsection{Universal locally acyclic sheaf}
For a finitely presented and separated map of schemes $f\colon X\to S$, there is a notion of the full subcategory of universally locally acyclic (ULA) sheaves
\[
D^{\ULA/S}(X,\Lambda)\subset D_c^b(X,\Lambda),
\]
see \cite[Definition 3.2]{HS} for a precise definition.
If $A\in D^{\ULA/S}(X,\Lambda)$, we say that $A$ is $f$-ULA or ULA over $S$.

This notion of universal local acyclicity has the following basic properties:
\begin{lemm}(\cite{HS})
Let $X$ be a scheme.
$A\in D^b_c(X,\Lambda)$ is $\rm{id}_X$-ULA if and only if $A$ is locally constant with perfect fibers.
\end{lemm}
\begin{proof}
See, for example, the statement of \cite[Lemma 4.3]{HS}.
\end{proof}
\begin{lemm}(\cite{HS})\label{lemm:ULApropsm}
Let $X\overset{g}{\to} Y\to S$ be finitely presented separated maps of schemes.
\begin{enumerate}
\item If $g$ is proper and $A\in D^{\ULA/S}(X,\Lambda)$, then $Rg_*A$ is in $D^{\ULA/S}(Y,\Lambda)$.
\item If $g$ is smooth and $A\in D^{\ULA/S}(Y,\Lambda)$, then $g^*A$ is in $D^{\ULA/S}(X,\Lambda)$.
\end{enumerate}
\end{lemm}
\begin{proof}
See the paragraph after \cite[Proposition 3.3]{HS}
\end{proof}
For a morphism $X\to S$ of perfect schemes and $A\in D^b_c(X,\Lambda)$, we say $A$ is ULA over $S$ if $A$ is ULA over $S'$ for some model $X'\to S'$ of $X\to S$, noting that $D^b_c(X,\Lambda)=D^b_c(X',\Lambda)$.
\begin{prop}\label{prop:jwLULA}
Let 
\[
j_w\colon \Fl_{G,\Divtil, w}\hookrightarrow \Fl_{G,\Divtil \leq w}
\]
be an inclusion.
Then $j_{w!}\Lambda\in D(\Fl_{G,\Divtil \leq w},\Lambda)$ is ULA over $\Divtil$.
\end{prop}
Here for a complex $A$ on an algebraic space $X$, we call $A$ ULA if its pullback to its atlas is ULA.
\begin{proof}
The proof is the same as \cite[Proposition VI.5.7]{FS}.
Consider the inclusion
\[
\wt{j}_w\colon \Fl_{w}\hookrightarrow \rm{Dem}_{\dot{w}}.
\]
By Lemma \ref{lemm:ULApropsm} (i), it suffices to show that $\wt{j}_{w!}\Lambda$ is ULA.
$\wt{j}_{w!}\Lambda$ can be resolved in terms of $i_{\dot{w}',\dot{w}*}\Lambda$ where 
\[
i_{\dot{w}',\dot{w}*}\colon \rm{Dem}_{\dot{w}'}\hookrightarrow \rm{Dem}_{\dot{w}}
\]
is a closed immersion corresponding to a subword $\dot{w}'$ of $\dot{w}$.
By the perfect smoothness of $\rm{Dem}_{\dot{w}'}\to \Divtil$ and by Lemma \ref{lemm:ULApropsm} (ii), the sheaf $i_{\dot{w}',\dot{w}*}\Lambda$ is ULA.
This gives the result.
\end{proof}
\begin{defi}
For 
\[
A\in D(\Hck_{G,S},\Lambda)^{\rm{bd}}, 
\]
we say $A$ is universally locally acyclic (ULA) over $S$ if its pullback to $D^b_c(\Gr_{G,S},\Lambda)^{\rm{bd}}$ is (when restricted to its support) universally locally acyclic (ULA) over $S$.
Write
\[
D^{\ULA}(\Hck_{G,S},\Lambda) \subset D(\Hck_{G,S},\Lambda)^{\rm{bd}}
\]
for the full subcategory of ULA sheaves.
\end{defi}
\begin{prop}(cf.\cite[Proposition VI.6.5]{FS})\label{prop:charaULA}
Let $S\to \Divtil$ be a perfect scheme over $\Divtil$.
Assume that $G$ is split.
An object $A\in D(\Hck_{G,S},\Lambda)^{\rm{bd}}$ is ULA over $S$ if and only if
\[
[\mu]^*A\in D(S,\Lambda)
\]
is locally constant with perfect fibers for all $\mu\in \bb{X}_*^+(T)$.
\end{prop}
\begin{proof}
The proof is the same as \cite[Proposition VI.6.5]{FS}.

First, we claim that if all $[\mu]^*A$ are locally constant with perfect fibers, then $A$ is ULA over $S$.
It reduces to the case where $S=\Divtil$ and $A=j_{\mu!}\Lambda$ where $j_{\mu}\colon \Gr_{G,\mu}\to \Gr_G$ is the inclusion.
By Proposition \ref{prop:mapFlGr} and Lemma \ref{lemm:ULApropsm} (i), it reduces to the case of the affine flag variety, which follows from Proposition \ref{prop:jwLULA}.

The converse follows by induction on the support of $A$.
In fact, the restriction of $A$ to a maximal Schubert cell $\Gr_{G,\mu}$ where $A$ is nonzero is ULA, and the stack
\[
[L^+G \backslash \Gr_{G,\mu}]
\]
can be seen as the classifying space of a smooth group over $\Divtil$, modulo pro-unipotent group.
Thus by Lemma \ref{lemm:ULApropsm} (ii), $[\mu]^*A$ is ULA, i.e. locally constant with perfect fibers.
Replacing $A$ by the cone of $j_{\mu!}j_\mu^* A\to A$ and using induction hypothesis, the result follows.
\end{proof}
\begin{cor}(cf.\cite[Corollary VI.6.6]{FS})\label{cor:ULAstabcompati}
Let $S\to \Divtil$ be a perfect scheme over $\Divtil$.
The full subcategory 
\[
D^{\ULA}(\Hck_{G,S},\Lambda)
\]
is stable under Verdier duality and $-\otimes_{\Lambda}^{\bb{L}}-$, $R\HHom_{\Lambda}(-,-)$, $j_!j^*$, $Rj_*j^*$, $j_!Rj^!$ and $Rj_*Rj^!$, where $j$ is a locally closed immersion of a Schubert cell.
Moreover, all of these operations commute with pullbacks in $S$.
\end{cor}
\begin{proof}
The proof is the same as \cite[Corollary VI.6.6]{FS}.
We may assume $G$ is split.
Stability under Verdier duality and compatibility with base change in $S$ is from \cite[Proposition 3.4]{HS}.
Let 
\begin{align*}
j_{\mu}&\colon \Gr_{G,S,\mu}\hookrightarrow \Gr_{G,S},\\
j_{\leq \mu}&\colon \Gr_{G,S,\leq \mu}\hookrightarrow \Gr_{G,S}
\end{align*}
be the inclusions.
Stability under $-\otimes_{\Lambda}^{\bb{L}}-$ and $j_{\mu!}j_{\mu}^*$ follows from the characterization in Proposition \ref{prop:charaULA}.
Since the map
\[
[\mu]\colon S\to [L^+_SG\backslash\Gr_{G,S,\mu}]
\]
is perfectly smooth modulo pro-unipotent group, and the Verdier dual of a locally constant sheaf with perfect fibers is locally constant with perfect fibers, one can show that if $[\mu]^*A$ is locally constant with perfect fibers for $A\in D([L_S^+G\backslash\Gr_{G,S,\mu}],\Lambda)$, then so is $[\mu]^*\bb{D}A$.
Therefore, $Rj_{\mu*}j_{\mu}^*=\bb{D}j_{\mu!}\bb{D}j_{\mu}^*$ preserves universal local acyclicity.

By Verdier duality, it follows that $Rj_{\mu*}Rj_{\mu}^!$ and $j_{\mu!}Rj_{\mu}^!$ preserve ULA sheaves.
It remains to show that for $A,B\in D^{\ULA}(\Hck_{G,S},\Lambda)$, it holds that $R\HHom(A,B)\in D^{\ULA}(\Hck_{G,S},\Lambda)$.
Consider the maps of stacks
\[
S\overset{[\mu]}{\to} [L^+_SG\backslash\Gr_{G,S,\mu}] \overset{j_{\mu}}{\hookrightarrow} [L^+_SG\backslash \Gr_{G,S}]=\Hck_{G,S}.
\]
By Verdier duality and the fact that $[\mu]$ is perfectly smooth modulo pro-unipotent group, it remains to show that 
\begin{align*}
[\mu]^*Rj_{\mu}^!R\HHom_{\Lambda}(A,B)
\end{align*}
is locally constant with perfect fibers.
It holds that
\begin{align*}
[\mu]^*Rj_{\mu}^!R\HHom_{\Lambda}(A,B) \cong R\HHom_{\Lambda}([\mu]^*j_{\mu}^*A,[\mu]^*Rj_{\mu}^!B).
\end{align*}
Since $[\mu]^*j_{\mu}^*A$ is locally constant with perfect fibers and so is $[\mu]^*Rj_{\mu}^!B$ by Verdier duality, it suffices to show that $R\HHom_{\Lambda}(-,-)$ on $S$ preserves locally constant sheaves with perfect fibers, and this is easy.
\end{proof}
\subsection{Derived category for Affine flag variety}
Let $S\to \Divtil$ be a perfect scheme over $\Divtil$.
Similarly, we define a category
\[
D(\AHck_{G,S},\Lambda)^{\bd}:=D_{\cal{I}}(\Fl_{G,S},\Lambda)^{\bd}, 
\]
where $\AHck_{G,S}$ means the affine Hecke stack
\[
\AHck_{G,S}=[\cal{I}\backslash \Fl_{G,S}].
\]
\begin{defi}
For 
\[
A\in D(\AHck_{G,S},\Lambda)^{\rm{bd}}, 
\]
we say $A$ is universally locally acyclic (ULA) over $S$ if its pullback to the support in $\Fl_{G,S}$ is universally locally acyclic (ULA) over $S$.
Write
\[
D^{\ULA}(\AHck_{G,S},\Lambda) \subset D(\AHck_{G,S},\Lambda)^{\rm{bd}}
\]
for the full subcategory of ULA sheaves.
\end{defi}
The similar results to Proposition \ref{prop:charaULA} and Corollary  \ref{cor:ULAstabcompati} hold for $\AHck$ as well.
\section{The proof of main theorem}
In this and the next section, let $k$ be an algebraic closure of $k_E$.
Consider the two morphism
\begin{align*}
(\pi,t)\colon \Spec k\to \Divtil,\\
(t,\pi)\colon \Spec k\to \Divtil,
\end{align*}
corresponding to a point $(\xi_1,\xi_2)=(\pi,t)$ and $(t,\pi)$, respectively.
We write 
\[
\Gr^{\rm{eq}}_{G,\Spec k}
\]
for the pullback of $\Gr_{G,\Divtil}$ under $(\pi,t)$ and 
\[
\Gr^{\rm{Witt}}_{G,\Spec k}
\]
for the pullback under $(t,\pi)$, and similarly for Hecke stacks and so on.
In fact, $\Gr^{\rm{eq}}_{G,\Spec k}$ is equal to the perfection of the usual affine Grassmannian in equal characteristics, and $\Gr^{\rm{Witt}}_{G,\Spec k}$ is equal to Zhu's Witt vector affine Grassmannian defined in \S \ref{ssc:reviewGrass} changing the base to $\Spec k$.
\begin{thm}\label{thm:maineqmixedtorsion}
There exists a canonical monoidal equivalence of categories
\[
D^{\ULA}(\Hck_{G,\Spec k}^{\rm{eq}},\Lambda)\to D^{\ULA}(\Hck_{G,\Spec k}^{\rm{Witt}},\Lambda),
\]
and similarly for $\AHck$.
\end{thm}
\begin{proof}
Let $\cal{O}$ be a coordinate ring of $\Divtil$ and $K$ its fractional field.
Write
\[
\ol{\cal{O}}_{(\pi,t)}\subset \ol{K}
\]
for the localization of the integral closure of $\cal{O}$ in $K$ at a prime lying above $(\pi,t)$.
We will show that the two pullback functors
\begin{align}
D^{\ULA}(\Hck_{G,\Spec \ol{\cal{O}}_{(\pi,t)}},\Lambda)&\overset{i^*}{\to} D^{\ULA}(\Hck^{\rm{eq}}_{G,\Spec k},\Lambda),\label{eqn:equivOkE}\\
D^{\ULA}(\Hck_{G,\Spec \ol{\cal{O}}_{(\pi,t)}},\Lambda)&\overset{j^*}{\to} D^{\ULA}(\Hck_{G,\Spec \ol{K}},\Lambda)\label{eqn:equivOKC}
\end{align}
are equivalences by a similar argument to \cite[Corollary VI.6.7]{FS} as follows:

We can first show that the category of locally constant sheaves with perfect fibers on $\Spec \ol{\cal{O}}_{(\pi,t)}$, $\Spec k$ and $\Spec \ol{K}$ are equivalent each other under the pullback functors, as these three categories are equivalent to the category of perfect $\Lambda$-complexes.

For a scheme $S$ over $\Spec k$, we write 
\[
D^{\rm{consta}}(S,\Lambda)
\]
for the essential image of the pullback functor
\[
D^b(\Spec k,\Lambda)\to D^b(S,\Lambda), 
\]
where $D^b(-,\Lambda)$ denotes the bounded derived category.
Then we can use the following lemma:
\begin{lemm}
Let $S$ and $X$ be schemes over $\Spec k$.
Assume that the pullback functor 
\[
D^b(\Spec k,\Lambda)\to D^{\rm{consta}}(S,\Lambda)
\]
is an equivalence and that $X$ is smooth over $\Spec k$.
Then the pullback functor 
\[
D^{\rm{consta}}(X,\Lambda)\to D^{\rm{consta}}(X\times_{\Spec k} S,\Lambda)
\]
is an equivalence.
\end{lemm}
\begin{proof}
The essential surjectivity is clear from the definition of $D^{\rm{consta}}(-,\Lambda)$.

It remains to show that for two bounded $\Lambda$-complexes $M,N$, 
\[
H^0(X,R\Hom(M,N))=H^0(X\times S,R\Hom(M,N)).
\]
Considering a diagram
\[
\xymatrix{
X\times_{\Spec k} S\ar[r]^p\ar[d]_b &X\ar[d]^a\\
S\ar[r]_q&\Spec k,
}
\]
it suffices to show that 
\[
Ra_*a^*\cong Ra_*Rp_*p^*a^*.
\]
This follows from the computation
\begin{align*}
Ra_*Rp_*p^*a^*&\cong Ra_*Rp_*b^*q^*\\
&\cong Ra_*a^*Rq_*q^*\\
&\cong Ra_*a^*,
\end{align*}
where the second isomorphism follows from the smoothness of $X$ and the third one from the assumption on $S$.
\end{proof}
Let us return to the proof of Theorem \ref{thm:maineqmixedtorsion}
Now it follows that the pullback functors
\begin{align*}
D^{\rm{consta}}((\bar{P}^{-}_{\mu})^n\times_{\Spec k_E} \Spec \ol{\cal{O}}_{(\pi,t)},\Lambda)&\underset{\sim}{\to} D^{\rm{consta}}((\bar{P}^{-}_{\mu})^n\times_{\Spec k_E} \Spec k,\Lambda),\\
D^{\rm{consta}}((\bar{P}^{-}_{\mu})^n\times_{\Spec k_E} \Spec \ol{\cal{O}}_{(\pi,t)},\Lambda)&\underset{\sim}{\to} D^{\rm{consta}}((\bar{P}^{-}_{\mu})^n\times_{\Spec k_E} \Spec \ol{K},\Lambda)
\end{align*}
are equivalences since these are equivalent to $D^{\rm{consta}}((\bar{P}^{-}_{\mu})^n\times_{\Spec k_E} \Spec k,\Lambda)$ via pullback functors.
Therefore we have equivalences
\begin{align*}
D^{\ULA}([\Spec k_E/\bar{P}^{-}_{\mu}]\times_{\Spec k_E} \Spec \ol{\cal{O}}_{(\pi,t)},\Lambda)&\underset{\sim}{\to} D^{\ULA}([\Spec k_E/\bar{P}^{-}_{\mu}]\times_{\Spec k_E} \Spec k,\Lambda),\\
D^{\ULA}([\Spec k_E/\bar{P}^{-}_{\mu}]\times_{\Spec k_E} \Spec \ol{\cal{O}}_{(\pi,t)},\Lambda)&\underset{\sim}{\to} D^{\ULA}([\Spec k_E/\bar{P}^{-}_{\mu}]\times_{\Spec k_E}  \Spec \ol{K},\Lambda)
\end{align*}
since these are the limit of the categories $D^{\rm{consta}}((\bar{P}^{-}_{\mu})^n\times_{\Spec k_E} (-),\Lambda)$ over $n$.
Thus, by Proposition \ref{prop:LGmu}, we have equivalences
\begin{align}
D^{\ULA}(\Hck_{G,\Spec \ol{\cal{O}}_{(\pi,t)},\mu},\Lambda)&\underset{\sim}{\overset{i^*}{\to}} D^{\ULA}(\Hck^{\rm{eq}}_{G,\Spec k,\mu},\Lambda),\label{eqn:equivOkEmu}\\
D^{\ULA}(\Hck_{G,\Spec \ol{\cal{O}}_{(\pi,t)},\mu},\Lambda)&\underset{\sim}{\overset{j^*}{\to}} D^{\ULA}(\Hck_{G,\Spec \ol{K},\mu},\Lambda)\label{eqn:equivOKCmu}
\end{align}
for all $\mu$.
From this, we can show the following claim:
\begin{quote}
For $A\in D^{\ULA}(\Hck_{G,\Spec \ol{\cal{O}}_{(\pi,t)}},\Lambda)$, 
\begin{align}\label{eqn:RGammaA}
R\Gamma(A)\cong R\Gamma(i^*A)\cong R\Gamma(j^*A).
\end{align}
\end{quote}
In fact, by induction on $\supp(A)$ using excision distinguished triangle, this reduces to the case $A\in D^{\ULA}(\Hck_{G,\Spec \ol{\cal{O}}_{(\pi,t)},\mu},\Lambda)$ for some $\mu$.
In this case, the claim follows from (\ref{eqn:equivOkEmu}) and (\ref{eqn:equivOKCmu}).

By (\ref{eqn:RGammaA}), the compatibility of $R\HHom(-,-)$ with base changes and the stability of universal local acyclicity under $R\HHom(-,-)$, we can show that the above two pullbacks are fully faithful, noting that
\[
\Hom(-,-)=H^0(R\Gamma(R\HHom(-,-))).
\]
A ULA sheaf is an extension of the sheaves in Schubert cells.
The equivalences (\ref{eqn:equivOkE}) and (\ref{eqn:equivOKC}) follow from the equivalences (\ref{eqn:equivOkEmu}) and (\ref{eqn:equivOKCmu}), and the fact that the sets of extensions are equivalent by full faithfulness.

By (\ref{eqn:equivOkE}) and (\ref{eqn:equivOKC}), we have an equivalence
\begin{align}
D^{\ULA}(\Hck^{\rm{eq}}_{G,\Spec k},\Lambda)\simeq  D^{\ULA}(\Hck_{G,\Spec \ol{K}},\Lambda). \label{eqn:equivkEKC}
\end{align}
By the same argument, we also have
\[
D^{\ULA}(\Hck^{\rm{Witt}}_{G,\Spec k},\Lambda)\simeq  D^{\ULA}(\Hck_{G,\Spec \ol{K}},\Lambda).
\]
Thus we obtain
\begin{align}
D^{\ULA}(\Hck^{\rm{eq}}_{G,\Spec k},\Lambda)\simeq D^{\ULA}(\Hck^{\rm{Witt}}_{G,\Spec k},\Lambda).\label{eqn:equivkEkE}
\end{align}
This equivalence is monoidal by Lemma \ref{lemm:convpullback}.
\end{proof}
In particular, setting $\Lambda=\Qlb$ in Theorem \ref{thm:maineqmixedtorsion} and 
noting that any object $A\in D^b_c(\AHck_{G,S},\Qlb)^{\bd}$ is ULA, we have the following result:
\begin{thm}\label{thm:maineqmixedQlb}
There exists a canonical equivalence of categories
\[
D^b_c(\Hck_{G,\Spec k}^{\rm{eq}},\Qlb)\overset{\sim}{\to} D^b_c(\Hck_{G,\Spec k}^{\rm{Witt}},\Qlb), 
\]
and similarly for $\AHck$.
\end{thm}
\begin{cor}(Derived Satake correspondence for Witt vector affine Grassmannian)\label{cor:DerivedSatbody}
There exists a canonical equivalence
\[
D^b_c(\Hck_{G,\Spec k}^{\Witt},\Qlb)\overset{\sim}{\longrightarrow} D^{\wh{G}}_{perf}(\Sym^{[]}(\wh{\fk{g}})).
\]
\end{cor}
\begin{proof}
Use the Derived Satake correspondence in equal characteristics (\cite{BF}) and Theorem \ref{thm:maineqmixedQlb}.
\end{proof}
\begin{cor}(Affine Hecke category for Witt vector affine Grassmannian)\label{cor:AffHckbody}
There exists a canonical equivalence
\[
D^b_c(\AHck_{G,\Spec k}^{\Witt},\Qlb)\overset{\sim}{\longrightarrow}\mathrm{DGCoh}^{\wh{G}}(\wt{N}\times_{\wh{\fk{G}}}\wt{N}).
\]
\end{cor}
\begin{proof}
Use the result in equal characteristics (\cite{Bez}) and Theorem \ref{thm:maineqmixedQlb}.
\end{proof}
\section{Perverse sheaves}
In this section, we will show that the equivalence in Theorem \ref{thm:maineqmixedtorsion} preserves perversity.
First, we need to define the perversity.
\begin{defi}\label{defi:pervalgsp}
Let $S$ be a scheme over $\bb{Z}[\frac{1}{\ell}]$.
Let $X$ be an algebraic space over $S$.
The relative perverse t-structure $({}^pD^{\leq 0},{}^pD^{\geq 0})$ on $D^b_c(X,\Lambda)$ is defined as follows:
Let $X'\to X$ be a \'{e}tale cover from a scheme $X'$.
For $A\in D^b_c(X,\Lambda)$,
\[
A\in {}^pD^{\leq 0}(X,\Lambda)\ \text{(resp.$A\in {}^pD^{\geq 0}(X,\Lambda)$)}
\]
if its pullback to $X'$ is in ${}^pD^{\leq 0}(X',\Lambda)$\ (resp.$A\in {}^pD^{\geq 0}(X',\Lambda)$) in the sense of \cite[Theorem 6.1]{HS}.

We call $A$ perverse if $A\in {}^pD^{\leq 0}(X,\Lambda)\cap {}^pD^{\geq 0}(X,\Lambda).$
\end{defi}
\begin{defi}\label{defprop:perv}
Let $S\to \Divtil$ be a perfect scheme over $\Divtil$.
For $A\in D(\Hck_{G,S},\Lambda)^{\rm{bd}}$, we write
\[
A\in {}^pD^{\leq 0}(\Hck_{G,S},\Lambda)^{\rm{bd}} \text{(resp. ${}^pD^{\geq 0}(\AHck_{G,S},\Lambda)^{\rm{bd}}$)}
\]
if its pullback to $D^b_c(\Gr_{G,S},\Lambda)$ is (when restricted to its support) in ${}^pD^{\leq 0}$ (resp. ${}^pD^{\geq 0}$) in the sense of Definition \ref{defi:pervalgsp}.

We call $A$ perverse if $A\in {}^pD^{\leq 0}(\Hck_{G,S},\Lambda)\cap {}^pD^{\geq 0}(\Hck_{G,S},\Lambda)$.
\end{defi}
By \cite[Theorem 6.1]{HS}, the perverse t-structure is fiberwise, and on fibers, it is equal to the usual (absolute) perverse t-structure.
Thus $A\in {}^pD^{\leq 0}(\Hck_{G,S},\Lambda)^{\rm{bd}}$  if and only if for all geometric points $\Spec C\to S$ and for all $\mu \in \bb{X}_*^+(T)$, the pullback of $A$ to 
\[
\Gr_{S,\Spec C,\mu}
\]
sits in cohomological degrees $\leq -\langle2\rho,\mu\rangle$, noting Proposition \ref{prop:LGmu}.
Moreover, let us define the notion of flat perversity:
\begin{defi}
Let $S\to \Divtil$ be a perfect scheme over $\Divtil$.
An object $A\in D(\Hck_{G,S},\Lambda)^{\rm{bd}}$ is flat perverse if $A\otimes_{\Lambda}^{\bb{L}} M$ is perverse for any $\Lambda$-module $M$.
\end{defi}
Write
\[
\Sat(\Hck_{G,S},\Lambda)\subset D^{\ULA}(\Hck_{G,S},\Lambda)^{\rm{bd}}
\]
for the full subcategory of flat perverse ULA sheaves.
\begin{prop}\label{prop:pervpreserve}
Let $S'\to S\to \Divtil$ be morphisms of perfect schemes.
The pullback functor
\[
f^*\colon D^{\ULA}(\Hck_{G,S},\Lambda)\to D^{\ULA}(\Hck_{G,S'},\Lambda)
\]
preserves the objects in $\Sat$.
The same holds for $\AHck$.
\end{prop}
\begin{proof}
This follows from \cite[Theorem 6.1(ii)]{HS} noting that $f^*$ commutes with tensor products.
\end{proof}
\begin{prop}(cf.\cite[Proposition VI.7.7]{FS})\label{prop:flatpervchar}
Assume $G$ is split and $\Lambda$ is a torsion ring.
Let $S\to \Divtil$ be morphisms of perfect schemes.
An object $A\in D(\Hck_{G,S},\Lambda)^{\rm{bd}}$ is flat perverse if and only if 
\[
R\pi_{T*}\rm{CT}_B(A)[\deg]\in D(S,\Lambda)
\]
is \'{e}tale locally on $S$ isomorphic to a finite projective $\Lambda$-module in degree 0.
Here the functor $\rm{CT}_B(-)[\deg]$ is the (shifted) hyperbolic localization functor defined just as in \cite[Corollary VI.3.5]{FS}, and $\pi_T\colon \Hck_{T,S}\to S$ is the projection.
\end{prop}
\begin{proof}
The proof is the same as \cite[Proposition VI.7.7]{FS}.
\end{proof}
Now we can show the following theorem:
\begin{thm}\label{thm:equivPerv}
The equivalence in Theorem \ref{thm:maineqmixedQlb} induces the equivalence
\begin{align}\label{eqn:pervHckequiv}
\Sat(\Hck_{G,\Spec k}^{\rm{eq}},\Lambda)\simeq \Sat(\Hck_{G,\Spec k}^{\rm{Witt}},\Lambda).
\end{align}
\end{thm}
\begin{proof}
By the standard argument, this reduces to the case where $\Lambda$ is a torsion ring.
By Proposition \ref{prop:pervpreserve} and the proof of Theorem \ref{thm:maineqmixedtorsion}, there exist fully faithful functors
\begin{align*}
\Sat(\Hck_{G,\Spec \ol{\cal{O}}_{(\pi,t)}},\Lambda)&\overset{f^*}{\hookrightarrow} \Sat(\Hck^{\rm{eq}}_{G,\Spec k},\Lambda),\\
\Sat(\Hck^{\rm{eq}}_{G,\Spec \ol{\cal{O}}_{(\pi,t)}},\Lambda)&\hookrightarrow \Sat(\Hck_{G,\Spec \ol{K}},\Lambda),\\
\Sat(\Hck_{G,\Spec \ol{\cal{O}}_{(t,\pi)}},\Lambda)&\hookrightarrow \Sat(\Hck^{\rm{Witt}}_{G,\Spec k},\Lambda),\\
\Sat(\Hck^{\rm{eq}}_{G,\Spec \ol{\cal{O}}_{(t,\pi)}},\Lambda)&\hookrightarrow \Sat(\Hck_{G,\Spec \ol{K}},\Lambda).
\end{align*}
We claim that these are equivalences.
The four proofs are the same, so we only deal with the first functor $f^*$.
For $A\in \Sat(\Hck^{\rm{eq}}_{G,\Spec k},\Lambda)$, there exists a $A'\in D^{\ULA}(\Hck_{G,\Spec \ol{\cal{O}}_{(\pi,t)}},\Lambda)$ unique up to isomorphism such that $f^*A'\cong A$.
The object $R\pi_{T*} \rm{CT}_B(A)[\deg]\in D(\Spec k,\Lambda)$ is isomorphic to a finite projective $\Lambda$-module in degree 0 by Proposition \ref{prop:flatpervchar}.
Since $R\pi_{T*} \rm{CT}_B(-)[\deg]$ preserves ULA sheaves, $R\pi_{T*} \rm{CT}_B(A')[\deg]$ is locally constant with perfect fibers.
As the hyperbolic localization commutes with pullbacks, it follows that $R\pi_{T*} \rm{CT}_B(A')[\deg]$ is isomorphic to a finite projective $\Lambda$-module in degree 0.
Thus $A'$ is flat perverse.
This implies that the functor $f^*$ is essentially surjective, and thus equivalence.
\end{proof}
\section{Affine Grassmannian over v-sheaf $\Div^{1,1}$ and symmetric monoidal structure} \label{scn:symm}
In this section, we prove the following theorem:
\begin{thm}\label{thm:symmonoidal}
The equivalence (\ref{eqn:pervHckequiv}) is symmetric monoidal.
\end{thm}
Here the symmetric monoidal structure on both sides of (\ref{eqn:pervHckequiv}) is defined as follows:

If $\Lambda$ is an algebraic extension $L$ of $\QQ_{\ell}$ (resp. its ring of integers $\cal{O}_L$), the symmetric monoidal structure is induced from that in the cases where $\Lambda$ is a finite extension $L'$ of $\QQ_{\ell}$ in $L$ (resp. its ring of integers $\cal{O}_{L'}$) by passing to the limit.
If $\Lambda$ is a finite extension $L$ of $\QQ_{\ell}$, then the symmetric monoidal structure is induced from that in the case where $\Lambda=\cal{O}_L$.
Moreover, if $\Lambda=\cal{O}_L$, the symmetric monoidal structure is induced from that in the cases where $\Lambda$ is a quotient ring of $\cal{O}_L$ by passing to the limit.
Thus it suffices to give the symmetric monoidal structure in the case where $\Lambda$ is a torsion ring.

As in \cite[Chapter VI]{FS}, there exists a symmetric monoidal structure on both sides of (\ref{eqn:pervHckequiv}).
More precisely, there exists a symmetric monoidal structure on 
\[
\Sat(\Hck_{G,\Div^1},\Lambda)
\]
coming from the fusion product as in \cite[\S VI.9]{FS}, where $\Div^1$ is a v-sheaf defined in \cite[Definition VI.1.1]{FS} (a moduli space of closed Cartier divisors on a Fargues--Fontaine curve).
Also, there exists a symmetric monoidal structure on 
\[
\Sat(\Hck_{G,\Spd C},\Lambda)
\]
where $C$ is a completion $\wh{\ol{E}}$ of an algebraic closure of $E$ such that the pullback functor along $\Spd C\to \Div^1$
\[
\Sat(\Hck_{G,\Div^1},\Lambda)\to \Sat(\Hck_{G,\Spd C},\Lambda),
\]
which is essentially surjective, is symmetric monoidal.
From the equivalence 
\[
\Sat(\Hck_{G,\Spd C},\Lambda)\simeq \Sat(\Hck^{\Witt}_{G,\Spd k},\Lambda)
\]
induced from the equivalence in \cite[Corollary VI.6.7]{FS}, we have a symmetric monoidal structure on 
\[
\Sat(\Hck^{\Witt}_{G,\Spd k},\Lambda)\simeq \Sat(\Hck_{G,\Spec k}^{\rm{Witt}},\Lambda).
\]
The same is true in the equal characteristic case.

We introduce a v-sheaf $\Div^{1,d}$ on the category of perfectoid spaces over $k$, which is a moduli space of sequences of divisors on a ``Fargues--Fontaine surface''.

We use the notions around adic spaces and diamonds.
Let $\Perf_k$ be the category of perfectoid spaces over $k$.

For $S=\Spa(R,R^+)\in \Perf_k$, put
\[
\underline{\cal{Y}}_S:=(\Spa W_{\cal{O}_E}(R^+)[[t]])\setminus V([\varpi]), 
\]
where $\varpi$ is a pseudo uniformizer of $R$.
We define a v-sheaf $\Div^{1,d}_{\underline{\cal{Y}}}\colon \Perf_k^{\rm{op}}\to \rm{Sets}$ by
\begin{align*}
&\Div^{1,d}_{\underline{\cal{Y}}}(S)\\
&=\left\{(D_0,D_1,\ldots,D_d)\relmiddle|\begin{array}{l}
\text{$D_0$ is a closed Cartier divisor of $\underline{\cal{Y}}_S$.}\\
\text{$D_1,\ldots, D_d$ are closed Cartier divisors of (a space defined by) $D_0$.}\\
\text{Each $D_i$ ($1\leq i\leq d$) is isomorphic to some $\bb{Z}_p[[t]]$-untilt $S^{\sharp}$ of $S$,}\\
\text{ and the induced map $S \cong D_i^{\dia} \hookrightarrow \underline{\cal{Y}}_S^{\dia}$ is a map over $S$.}
\end{array} \right\}.
\end{align*}
for $S=\Spa(R,R^+)\in \Perf_k$.
Define a v-sheaf $\ol{\Div}^{1,d}_{\underline{\cal{Y}}}$ by
\[
\ol{\Div}^{1,d}_{\underline{\cal{Y}}}=\Div^{1,d}_{\underline{\cal{Y}}}/\Sigma_d
\]
 where $\Sigma_d$ is the symmetric group of degree $d$ acting as permutations of $D_1,\ldots,D_d$.
Let $\cal{O}_{D_0}$ be a structure sheaf of the closed subvariety $D_0$ of $\underline{\cal{Y}}_S$.
For $(D_0,\ldots,D_d)\in \Div^{1,d}_{\underline{\cal{Y}}}(S)$, write 
\[
B^+_{D_0,D_1,\ldots,D_d}(S)
\]
for the completion of $\cal{O}_{D_0}$ along the product $\cal{I}_{D_1}\ldots \cal{I}_{D_d}$ of the ideal sheaves corresponding to $D_1,\ldots,D_d$.
Put
\[
B_{D_0,D_1,\ldots,D_d}(S)=B^+_{D_0,\ldots,D_d}(S)\left[\frac{1}{\cal{I}_{D_1}\ldots \cal{I}_{D_d}}\right]
\]
and define the presheaves $L^+_{\Div^{1,d}}G$, $L_{\Div^{1,d}}G$ by 
\begin{align*}
L^+_{\Div^{1,d}_{\underline{\cal{Y}}}}G(S)&=G(B^+_{D_0,\ldots,D_d}(S)),\\
L_{\Div^{1,d}_{\underline{\cal{Y}}}}G(S)&=G(B_{D_0,\ldots,D_d}(S))
\end{align*}
for $S=\Spa(R,R^+)\in \Perf_k$.
It is obvious that $L^+_{\Div^{1,d}_{\underline{\cal{Y}}}}G$ and $L_{\Div^{1,d}_{\underline{\cal{Y}}}}G$ descend to the presheaves over $\ol{\Div}^{1,d}_{\underline{\cal{Y}}}$, and they are denoted by $L^+_{\ol{\Div}^{1,d}_{\underline{\cal{Y}}}}G$ and $L_{\ol{\Div}^{1,d}_{\underline{\cal{Y}}}}G$, respectively.
We also define the affine Grassmannian and local Hecke stack over $\ol{\Div}^{1,d}_{\underline{\cal{Y}}}$ by 
\begin{align*}
\Gr_{G,\ol{\Div}^{1,d}_{\underline{\cal{Y}}}}&:=[L_{\ol{\Div}^{1,d}_{\underline{\cal{Y}}}}G/L^+_{\ol{\Div}^{1,d}_{\underline{\cal{Y}}}}G],\\
\Hck_{G,\ol{\Div}^{1,d}_{\underline{\cal{Y}}}}&:=[L^+_{\ol{\Div}^{1,d}_{\underline{\cal{Y}}}}G\backslash L_{\ol{\Div}^{1,d}_{\underline{\cal{Y}}}}G/L^+_{\ol{\Div}^{1,d}_{\underline{\cal{Y}}}}G].
\end{align*}
For a small v-stack $S$ with a map $S\to \ol{\Div}^{1,d}_{\underline{\cal{Y}}}$, 
\begin{align*}
\Gr_{G,S}:=\Gr_{G,\ol{\Div}^{1,d}_{\underline{\cal{Y}}}}\times_{\ol{\Div}^{1,d}_{\underline{\cal{Y}}}}S,
\Hck_{G,S}:=\Hck_{G,\ol{\Div}^{1,d}_{\underline{\cal{Y}}}}\times_{\ol{\Div}^{1,d}_{\underline{\cal{Y}}}}S.
\end{align*}
By the same argument as \cite[Proposition VI.1.7, Proposition VI.1.9]{FS}, $\Gr_{G,S}$ is a small v-sheaf and $\Hck_{G,S}$ is a small v-stack.

For a perfect scheme $X$ of characteristic $p$, we define a small v-sheaf $X^{\diamondsuit}$ as in \cite[\S 18.3]{SW} (not as in \cite[\S 27]{Sch}).
That is, put $(\Spec R)^{\diamondsuit}=\Spd(R,R)$ and define $X^{\diamondsuit}$ for general $X$ by gluing.
There is a canonical functor of sites
\[
\alpha_X\colon X^{\diamondsuit}_{v}\to X_{\et}.
\]
The pullback along this defines a functor
\begin{align}\label{eqn:pbdiamond}
\alpha_X^*\colon D(X_{\et},\Lambda)\to D_{\et}(X^{\diamondsuit},\Lambda).
\end{align}

For a perfect $k$-algebra $A$ and a map $\Spec A\to \Divtil$ corresponding to $(\xi_1,\xi_2)\in W_{\cal{O}_E}(A)[[t]]$ and a map $\Spa(R,R^+)\to \Spa(A,A)$ from a perfectoid space $\Spa(R,R^+)\in\Perf$, we can define a map $\Spa(R,R^+) \to \Div^{1,1}$ corresponding to $(D_0,D_1)$, where $D_0$ is a divisor generated by an image of $\xi_1$ under
\[
W_{\cal{O}_E}(A)[[t]]\to W_{\cal{O}_E}(R^+)[[t]]\left[\frac{1}{\varpi}\right]
\]
and $D_1$ is a divisor generated by an image of $\xi_2$ under the same map (composed with a quotient map by $\cal{I}_{D_0}$).
From this, we get a map
\[
(\Divtil)^{\diamondsuit}\to \Div^{1,1}_{\underline{\cal{Y}}}.
\]
From the map $\Spec k\to \Divtil$ corresponding to $(t,\pi)\in \Divtil(k)$, we obtain the composition
\[
(t,\pi)\colon \Spd k\to (\Divtil)^{\diamondsuit}\to \Div^{1,1}_{\underline{\cal{Y}}}.
\]
Let $\ol{K}$ be an algebraic closure of $K(\Divtil)$ and let $\ol{\cal{O}}_{(t,\pi)}$ be a localization of the integral closure of $\cal{O}(\Divtil)$ in $\ol{K}$ at a prime lying over the prime corresponding to $(t,\pi)\in \Divtil(k)$.
We similarly obtain the maps
\begin{align*}
\Spd \ol{\cal{O}}_{(t,\pi)}&\to (\Divtil)^{\diamondsuit}\to \Div^{1,1}_{\underline{\cal{Y}}},\\
\Spd \ol{K}&\to (\Divtil)^{\diamondsuit}\to \Div^{1,1}_{\underline{\cal{Y}}}.
\end{align*}
Here we endow $k$, $\ol{\cal{O}}_{(t,\pi)}$ and $\ol{K}$ with discrete topologies.
\begin{proof}[Proof of Theorem \ref{thm:symmonoidal}]
By the standard argument, this reduces to the case where $\Lambda$ is a torsion ring.

It suffices to show that the isomorphism
\[
\xymatrix{
\Sat(\Hck^{\Witt}_{G,\Spec k})\ar@{<-}[r]^-{i^*}_-{\sim}&\Sat(\Hck_{G, \Spec \ol{\cal{O}}_{(t,\pi)}})\ar[r]^-{j^*}_-{\sim}&\Sat(\Hck_{G,\Spec \ol{K}})
}
\]
in Theorem \ref{thm:equivPerv} is symmetric monoidal with respect to the above symmetric monoidal structure on $\Sat(\Hck^{\Witt}_{G,\Spec k})$ and $\Sat(\Hck_{G,\Spec \ol{K}})$.
We have the following diagram:
\begin{align}\label{eqn:diagalpha}
\begin{gathered}
\xymatrix{
\Sat(\Hck^{\Witt}_{G,\Spec k})\ar@{<-}[r]^-{i^*}_-{\sim}\ar[d]_{\alpha^*}&\Sat(\Hck_{G,\Spec \ol{\cal{O}}_{(t,\pi)}})\ar[r]^-{j^*}_-{\sim}\ar[d]_{\alpha^*}&\Sat(\Hck_{G,\Spec \ol{K}})\ar[d]_{\alpha^*}\\
\Sat(\Hck^{\Witt}_{G,\Spd k})\ar@{<-}[r]^-{(i^{\diamondsuit})^*}&\Sat(\Hck_{G, \Spd \ol{\cal{O}}_{(t,\pi)}})\ar[r]^-{(j^{\diamondsuit})^*}&\Sat(\Hck_{G,\Spd \ol{K}}).
}
\end{gathered}
\end{align}
This diagram is commutative by definition (cf. \cite[Proposition 27.1]{Sch}).
By the argument for $\ol{K}$ instead of $k$, there exists a symmetric monoidal structure on $\Sat(\Hck_{G,\Spec \ol{K}})$.
The left vertical functor is an equivalence, and replacing $k$ by $\ol{K}$, the right vertical functor is also an equivalence.
Furthermore, we can show that the bottom two functors are equivalences, by the same argument as the equivalence of $i^* $ and $j^*$.
Here we use the equivalence
\[
D^{\ULA/\Spd \ol{\cal{O}}_{(t,\pi)}}(\Spd \ol{\cal{O}}_{(t,\pi)},\Lambda)\simeq D_{\perf}(\Lambda), 
\]
where the right-hand side is the category of perfect $\Lambda$-complexes.

Moreover, define the following rings:
\begin{align*}
C&:=\wh{\ol{E}} \text{(: the completion of an algebraic closure of $W_{\cal{O}_E}(k)[1/p]$.)}\\
\cal{O}_C&\text{: the completion of the integral closure of $W_{\cal{O}_E}(k)$ in $C$.}\\
CC&\text{: the completion of an algebraic closure of $W_{\cal{O}_E}(\ol{K})[1/p]$.}\\
\cal{O}C&\text{: the completion of the integral closure of $W_{\cal{O}_E}(\ol{K})$ in $CC$.}\\
C\cal{O}&\text{: the completion of the integral closure of $W_{\cal{O}_E}(\ol{\cal{O}}_{(t,\pi)})[1/p]$ in $CC$.}\\
\cal{OO}&\text{: the completion of the integral closure of $W_{\cal{O}_E}(\ol{\cal{O}}_{(t,\pi)})$ in $CC$.}
\end{align*}
Then we have the following diagram:
\begin{align}\label{eqn:diagexc}
\begin{gathered}
\xymatrix{
\Sat(\Hck^{\Witt}_{G,\Spd k})\ar@{<-}[r]\ar@{<-}[d]&\Sat(\Hck_{G,\Spd \ol{\cal{O}}_{(t,\pi)}})\ar[r]\ar@{<-}[d]&\Sat(\Hck_{G,\Spd \ol{K}})\ar@{<-}[d]\\
\Sat(\Hck_{G,\Spd \cal{O}_C})\ar@{<-}[r]&\Sat(\Hck_{G,\Spd \cal{OO}})\ar[r]&\Sat(\Hck_{G,\Spd \cal{O}C}).
}
\end{gathered}
\end{align}
The left vertical functor is an equivalence by \cite[Corollary VI.6.7]{FS} and the right vertical functor is also an equivalence replacing $k$ by $\ol{K}$.
Also, the middle vertical functor is an equivalence by the same argument as \cite[Corollary VI.6.7]{FS}, using the equivalence
\[
D^{\ULA/\Spd \cal{OO}}(\Spd \cal{OO},\Lambda)\simeq D_{\perf}(\Lambda).
\]
Furthermore, we have the following diagram:
\begin{align}\label{eqn:diagpb}
\begin{gathered}
\xymatrix{
\Sat(\Hck_{G,\Spd \cal{O}_C})\ar@{<-}[r]\ar@{<-}[d]&\Sat(\Hck_{G,\Spd \cal{O}\cal{O}})\ar[r]\ar@{<-}[d]&\Sat(\Hck_{G,\Spd \cal{O}C})\ar@{<-}[d]\\
\Sat(\Hck_{G,\Spd W_{\cal{O}_E}(k)})\ar@{<-}[r]&\Sat(\Hck_{G,\Spd W_{\cal{O}_E}(\ol{\cal{O}}_{(t,\pi)})})\ar[r]&\Sat(\Hck_{G,\Spd W_{\cal{O}_E}(\ol{K})}),
}
\end{gathered}
\end{align}
where the vertical maps are the pullback functors.
We can show that all the functors in (\ref{eqn:diagpb}) are equivalences by the same argument as above.
Here we use the equivalence
\[
D^{\ULA/\Spd W_{\cal{O}_E}(\ol{\cal{O}}_{(t,\pi)})}(\Spd W_{\cal{O}_E}(\ol{\cal{O}}_{(t,\pi)}),\Lambda)\simeq D_{\perf}(\Lambda)
\]
and similar equivalences for the other bases appearing in (\ref{eqn:diagpb}).

From the diagram (\ref{eqn:diagalpha}), (\ref{eqn:diagexc}) and (\ref{eqn:diagpb}), it suffices to show that there exists a symmetric monoidal structure on 
\[
\Sat(\Hck_{G,\Spd W_{\cal{O}_E}(\ol{\cal{O}}_{(t,\pi)})})
\]
such that the pullback functors to $\Sat(\Hck_{G,\Spd W_{\cal{O}_E}(k)})$ and $\Sat(\Hck_{G,\Spd W_{\cal{O}_E}(\ol{K})})$ are symmetric monoidal.
A symmetric monoidal structure on $\Sat(\Hck_{G,\Spd W_{\cal{O}_E}(\ol{\cal{O}}_{(t,\pi)})})$ can be defined via fusion product by the same argument as \cite[\S VI.9]{FS}.
Here we use the fact that the map
\[
\Spd W_{\cal{O}_E}(\ol{\cal{O}}_{(t,\pi)})\to \Spd \ol{\cal{O}}_{(t,\pi)}
\]
is $\ell$-cohomologically smooth since there is an isomorphism
\[
\Spd W_{\cal{O}_E}(\ol{\cal{O}}_{(t,\pi)})\cong \Spd W_{\cal{O}_E}(k)\times \Spd \ol{\cal{O}}_{(t,\pi)}
\]
and the map $\Spd W_{\cal{O}_E}(k)\to \Spd k$ is $\ell$-cohomologically smooth by \cite[Corollary IV.2.34]{FS}.
Then the two pullback functors
\begin{align*}
\Sat(\Hck_{G,\Spd W_{\cal{O}_E}(\ol{\cal{O}}_{(t,\pi)})})&\to \Sat(\Hck_{G,\Spd W_{\cal{O}_E}(k)}), \text{and}\\
\Sat(\Hck_{G,\Spd W_{\cal{O}_E}(\ol{\cal{O}}_{(t,\pi)})})&\to \Sat(\Hck_{G,\Spd W_{\cal{O}_E}(\ol{K})})
\end{align*}
are symmetric monoidal.
In fact, by the definition of fusion product, this reduces to the compatibility of exterior tensor products with pullbacks, which is obvious.
\end{proof}
\bibliographystyle{test}
\bibliography{DerivedSatakeBib}
\end{document}